\documentclass[preprint,12pt]{elsarticle}

\usepackage{amssymb}
\usepackage{amsmath}
\usepackage{amsthm}
\usepackage{pdflscape}
\newtheorem{corollary}{Corollary}
\newtheorem{definition}{Definition}
\newtheorem{example}{Example}
\newtheorem{lemma}{Lemma}
\newtheorem{nono-proof}{Proof}
\newtheorem{proposition}{Proposition}
\newtheorem{remark}{Remark}
\newtheorem{theorem}{Theorem}
\usepackage{lineno}
\journal{Journal of Mathematical Analysis and Applications}

\begin{document}

\begin{frontmatter}



\title{Towards a complete characterization of indicator variograms and madograms} %% Article title

\author[label1,label2]{Xavier Emery\corref{cor1}} 
\ead{xemery@ing.uchile.cl}
\author[label3]{Christian Lantuéjoul} 
\ead{christian.lantuejoul@minesparis.psl.eu}
\author[label1,label2]{Nadia Mery}
\ead{nmeryg@uchile.cl}
\author[label4]{Mohammad Maleki} 
\ead{mohammad.maleki@ucn.cl}
\cortext[cor1]{Corresponding author.}
%% Author name

\affiliation[label1]{organization={Universidad de Chile, Department of Mining Engineering},
             addressline={Avenida Beauchef 850},
             city={santiago},
             postcode={8370448},
             country={Chile}}

\affiliation[label2]{organization={Universidad de Chile, Advanced Mining Technology Center},
             addressline={Avenida Beauchef 850},
             city={santiago},
             postcode={8370448},
             country={Chile}}

\affiliation[label3]{organization={Mines ParisTech, PSL University, Centre de Géosciences},
             addressline={35 Rue Saint-Honoré},
             city={Fontainebleau},
             postcode={77300},
             country={France}}

\affiliation[label4]{organization={Universidad Católica del Norte, Department of Metallurgical and Mining Engineering},
             addressline={Angamos 0610},
             city={Antofagasta},
             postcode={1270709},
             country={Chile}}
             
%% Abstract
\begin{abstract}
Indicator variograms and madograms are structural tools used in many disciplines of the natural sciences and engineering to describe random sets and random fields. To date, several necessary conditions are known for a function to be a valid indicator variogram but, except for intractable corner-positive inequalities, a complete characterization of indicator variograms is missing. Likewise, only partial characterizations of madograms are known. This paper provides novel necessary and sufficient conditions for a given function to be the variogram of an indicator random field with constant mean value or to be the madogram of a random field, and establishes under which conditions these two families of functions coincide. Our results apply to any set of points where the random field is defined and rely on distance geometry and Gaussian random field theory.
\end{abstract}



%% Keywords

\end{frontmatter}


\section{Introduction}
\label{sec1}

\subsection{Random Sets, Indicator Variograms and Madograms}

Random sets are stochastic models whose realizations are subsets of a given metrizable space $\mathbb{X}$, e.g. a Euclidean space, a sphere, or a graph. The theory of random sets has been formalized by Matheron \citep{Matheron1975} and is of interest in probability, stochastic geometry, mathematical morphology, and geostatistics \citep{Serra1982, Lantuejoul2002, Chiu2013, Molchanov2017, Jeulin2021}.   

{This theory requires equipping the space $\mathbb{X}$ with the hit-and-miss (Fell) topology and focusing on closed subsets of $\mathbb{X}$. In such a case, a random closed set is fully characterized by its \emph{hitting functional}, which gives the probability that any compact (i.e., closed and bounded) set overlaps with the random closed set. An alternative characterization of a random closed set is via its indicator function, that is, the random field equal to $1$ inside the random set and to $0$ outside. Under an assumption of \emph{separability} (see \ref{app_background} for technical details), the random set is also completely characterized by all the finite-dimensional distributions of its indicator function \citep{Matheron1975}.}

When considering only two points in $\mathbb{X}$, a simplification that is often made in geostatistics to ease the modeling of the random set based on observations at finitely many points in $\mathbb{X}$, the finite-dimensional distribution is equivalent to the covariance function or the variogram. The latter functions have been recognized as tools to describe the geometry of the set boundary \citep{Emery2011} and are widely used as an input for indicator kriging \citep{Journel1983} or for simulating random sets with a given spatial correlation structure, e.g., via sequential indicator \citep{Alabert1987} or stochastic optimization \citep{Yeong1998, Jiao2008} algorithms. Application fields for indicator covariances and variograms, indicator kriging and indicator simulation include telecommunications \citep{McMillan1955}, materials science \citep{Jeulin2000, Torquato2002}, ecology \citep{Rossi1992}, environmental science \citep{Webster1994}, soil science \citep{He2009}, hydrology \citep{Western1998}, geology \citep{Jones2001}, petroleum reservoir modeling \citep{Guo2010, Cao2014}, and mineral deposit modeling \citep{Maleki2017}, among others.

An example of random set that is recurrent in these applications is the ``excursion set'' of a real-valued random field $Z$, that is, the region where the random field exceeds a certain threshold $\lambda$:
$$ g_\lambda (x,y) = \frac{1}{2} var \bigl\{ \mathsf{1}_{Z(x) \geq \lambda} - \mathsf{1}_{Z(y) \geq \lambda} \bigr\}, \qquad x,y \in \mathbb{X}, $$
where $\mathsf{1}_{Z(x) \geq \lambda} = 1$ if $Z(x) \geq \lambda$, $0$ otherwise. 
By letting the threshold vary, one obtains information on how fast the excursion sets associated with large thresholds lose their spatial correlation \citep{chiles_delfiner_2012}.

If the marginal distribution of $Z(x)$ is the same for all the points $x \in \mathbb{X}$, the integral of all the excursion set indicator variograms coincides with the so-called madogram, or first-order variogram, of $Z$ \citep{Matheron1989}:
\begin{equation*}
    \int_{-\infty}^{+\infty} g_\lambda (x,y) \, {\rm d} \lambda = \gamma^{(1)} (x,y), \qquad x,y \in \mathbb{X},
\end{equation*}
where
\begin{equation*}
    \gamma^{(1)}(x,y) = \frac{1}{2} \mathbb{E} \bigl\{ \vert Z(x) - Z (y) \vert \bigr\}, \quad x,y \in \mathbb{X}.
\end{equation*}

The above identity shows the close relationship between indicator variograms and madograms. The latter are used in spatial statistics to diagnose the fractal character of the realizations of a random field \citep{chiles_delfiner_2012}, and in spatial extreme value theory to characterize the extremal coefficient of max-stable random fields \citep{naveau2009} or the spatial dependence of extreme values \citep{bacro2010}.

\subsection{The Need for Consistent Models}

It is known \citep{Matheron1989} that every indicator variogram is a madogram, and that every madogram is a variogram. Conversely, one can wonder whether every variogram is a madogram and every madogram is (proportional to) an indicator variogram. Regarding the first question, the answer is negative. For instance, in Euclidean spaces, an isotropic variogram with a smooth behavior near the origin (e.g. a Gaussian variogram) cannot be a madogram. Regarding the second question, the answer is unknown. 

\medskip

The following example shows what is incurred when an arbitrary variogram is used in place of a genuine indicator variogram. Let $A$ be a separable random set 
defined on $\mathbb{X}$. It is characterized by the family of probabilities 
$$ P \bigl( B_U^V \bigr) = P \{ U \subset A, V \subset A^c \} \qquad U,V \subset \mathbb{X}, $$
with $A^c$ the complement of $A$ and $P(B_U)$ as a shorthand notation for $P(B_U^{\emptyset})$. 

Let $g$ be the variogram of the indicator of $A$ and assume that $A$ is ``autodual'', i.e, $A$ and $A^c$ have the same distributions, which means that $ P(B_U^V) = P(B_V^U)$ for all subsets $U$ and $V$ of $\mathbb{X}$ \citep{Matheron1993}. For any $x,y,z \in \mathbb{X}$, this implies $P(B_{x}) = \frac{1}{2}$, $P(B_{x,y}) = 
\frac{1}{2} - g(x,y)$ and $P(B_{x,y,z}) = \frac{1}{2} \bigl( 1 - g(x,y) - g(x,z) - g(y,z) \bigr) $, the latter equality stemming 
from Euler-Poincaré formula. Of paramount importance is the fact that all the trivariate distributions are characterized by the indicator variogram \citep{Lantuejoul2002}. In particular,  
$$ P(B_{x,z}^y) = P(B_{x,z}) - P(B_{x,y,z}) = \frac{1}{2} \bigl( g(x,y) - g(x,z) + g(y,z) \bigr), \quad x, y, z \in \mathbb{X}. $$

Now, here comes the problem: if $g$ is an invalid model that violates the triangle inequality $g(x,z) \leq g(x,y) + g(y,z)$, then $P(B_{x,z}^y) < 0$. 

Similar inconsistencies (negative probabilities of some events involving three points) were pointed out by Dubrule \cite{Dubrule2017} when $\mathbb{X}$ is the Euclidean space and $g$ is an isotropic Gaussian, cubic, or cardinal sine variogram, which does not fulfill the triangle inequality.

\subsection{Aims and Scope}

The aim of this paper is to get a better understanding of how the sets of indicator variograms and madograms are organized within the set of variograms. The outline is as follows: Section \ref{sec2} reviews a number of known properties satisfied by variograms, indicator variograms, and madograms. The main new theoretical results that characterize indicator variograms and madograms are then given in Section \ref{sec3}. Section \ref{sec:examples} exhibit examples of valid parametric families of indicator variograms in Euclidean spaces, spheres, graphs, and in any abstract set of points. A general discussion and conclusions are presented in Sections \ref{sec:discussion} and \ref{sec:conclusions}. Background material and complementary results are provided in \ref{app_background} to \ref{excursion}, while the proofs of lemmas, propositions, theorems, corollaries, and examples, are deferred to \ref{app_proof}. 

\medskip
To achieve generality in our presentation, minimal assumptions will be made on the set $\mathbb{X}$ on which the random field is defined and on the random field itself. For indicator variograms, the only requirement will be that of \emph{first-order stationarity}, i.e., the indicator random field has a constant mean value (see \ref{app_background}). When linking madograms and indicator variograms, we will require the random field to have a stationary marginal distribution, i.e., the same marginal distribution for all the points in $\mathbb{X}$. 

\section{Review of Known Results}

Throughout this section, $\mathbb{X}$ denotes an abstract set of points. 

\label{sec2}

\subsection{Variograms and Conditional Negative Semidefiniteness}

The variogram of a random field $Z$ on $\mathbb{X}$ is the mapping $\gamma: \mathbb{X} \times \mathbb{X} \to [0,+\infty)$ such that
\begin{equation*}
    \gamma(x,y) = \frac{1}{2} var\{ Z(x) - Z(y) \}, \quad x, y \in \mathbb{X},
\end{equation*}
which exists as soon as the increments of $Z$ have a finite variance. If, furthermore, the random field is first-order stationary, then
\begin{equation*}
    \gamma(x,y) = \frac{1}{2} \mathbb{E}\{ [Z(x) - Z(y)]^2 \}, \quad x, y \in \mathbb{X}.
\end{equation*}

It is well-known \citep{Schoenberg1938, chiles_delfiner_2012} that the class of variograms coincides with the class of conditionally negative semidefinite mappings that are symmetric and vanish on the diagonal of $\mathbb{X} \times \mathbb{X}$, i.e.: 
\begin{equation}
\label{condneg0}
    \gamma(x,x) = 0, \quad x \in \mathbb{X},
\end{equation}
\begin{equation}
\label{condneg1}
    \gamma(x,y) = \gamma(y,x), \quad x, y \in \mathbb{X},
\end{equation}
and satisfy negative-type inequalities:
\begin{equation}
\label{condneg}
    \sum_{k=1}^n \sum_{\ell=1}^n \lambda_k \, \lambda_\ell \, \gamma(x_k,x_\ell) \leq 0
\end{equation}
for all $x_1,\ldots,x_n \in \mathbb{X}$ and $\lambda_1,\ldots,\lambda_n \in \mathbb{R}$ such that $\sum_{k=1}^n \lambda_k = 0$.

\subsection{Indicator Variograms}

Not every symmetric and conditionally negative semidefinite mapping is eligible for an indicator random field: The class of indicator variograms is a strict subclass of the mappings satisfying Eqs. (\ref{condneg0}) to (\ref{condneg}). 
This raises two questions: (1) given an indicator random field on $\mathbb{X}$, which properties must satisfy its variogram? and (2) given a mapping $g: \mathbb{X} \times \mathbb{X} \to [0, +\infty)$, can it be the variogram of an indicator random field on $\mathbb{X}$? This last question, known as reconstruction or realizability problem, is of interest in diverse fields of application \citep{Torquato2002, Lachieze2015}.

Concerning the first question, several additional inequalities that must fulfill a first-order stationary indicator variogram have been established in the literature. In particular, given a set of $n$ points $x_1,\ldots,x_n$ in $\mathbb{X}$, the following inequalities hold:
\begin{enumerate}
    \item Upper-bound inequality:
    \begin{equation}
    \label{upperbound}
           g(x_k,x_\ell) \leq \frac{1}{2}, \quad k, \ell \leq n.
    \end{equation}

    \item Triangle inequality \citep{Matheron1989, Armstrong1992}:
    \begin{equation}
    \label{triang}
           g(x_k,x_m) + g(x_\ell,x_m) \geq g(x_k,x_\ell), \quad k, \ell, m \leq n.
    \end{equation}

    \item Matheron's inequalities \citep{Matheron1993}: for all $\lambda_1,\ldots,\lambda_n \in \{-1,0,1\}$ such that $\sum_{k=1}^n \lambda_k = 1$,
    \begin{equation}
    \label{1clique}
           \sum_{k=1}^n \sum_{\ell=1}^n \lambda_k \, \lambda_\ell \, g(x_k,x_\ell) \leq 0.
    \end{equation}

    \item Shepp's inequalities \citep{Shepp1963}: for all odd integer $n$ and all $\lambda_1,\ldots,\lambda_n \in \{-1,1\}$,
    \begin{equation}
    \label{shepp63}
           \sum_{k=1}^n \sum_{\ell=1}^n \lambda_k \, \lambda_\ell \, g(x_k,x_\ell) \leq \frac{1}{4} \left(\sigma^2(\boldsymbol{\lambda})-1 \right),
    \end{equation}
    where $\boldsymbol{\lambda} = (\lambda_1,\ldots,\lambda_n)$ and $\sigma(\boldsymbol{\lambda}) = \sum_{k=1}^n \lambda_k$.

    \item Corner-positive inequalities \citep{McMillan1955}: for any $n > 1$ and any corner-positive matrix $[\lambda_{k\ell}]_{k,\ell=1}^n$ (i.e., a matrix such that $\sum_{k=1}^n \sum_{\ell=1}^n \varepsilon_k \, \varepsilon_\ell \, \lambda_{k\ell} \geq 0$ for any $(\varepsilon_1,\ldots,\varepsilon_n) \in \{-1,1\}^n$), one has
    \begin{equation}
    \label{corner}
           \sum_{k=1}^n \sum_{\ell=1}^n \lambda_{k\ell} [1 - 4 g(x_k,x_\ell)] \geq 0.
    \end{equation}
\end{enumerate}

Matheron's inequalities (\ref{1clique}) include the triangle inequality and were conjectured to be not only necessary, but also sufficient to characterize a first-order stationary indicator variogram \citep{Matheron1993}. 
Regrettably, such a conjecture is not true \citep{Lachieze2015}. Shepp's inequalities (\ref{shepp63}) are stronger, but it is not clear whether or not they are sufficient. In contrast, the corner-positive inequalities (\ref{corner}) have been shown to be necessary and sufficient, see Shepp \citep{Shepp1967} for the case $\mathbb{X}= \mathbb{Z}$, Masry \citep{Masry1972} for the case $\mathbb{X}= \mathbb{R}$ and Quintanilla \citep{Quintanilla2008} for the case $\mathbb{X}= \mathbb{R}^N$ with $N \geq 1$. Interestingly, this last author reduced the check of (\ref{corner}) on each finite population of points to a finite set of inequalities; however, this corner-positive criterion is effective only if $\mathbb{X}$ is made of finitely many points (e.g., a finite graph).

\medskip

An interesting subclass of indicator variograms are the median indicator variograms of Gaussian random fields:

\begin{proposition}[McMillan \citep{McMillan1955}]
\label{prop01}
    Let $\rho$ be a correlation function on $\mathbb{X} \times \mathbb{X}$, i.e., a positive semidefinite mapping that is equal to $1$ on the diagonal of $\mathbb{X} \times \mathbb{X}$. Then, the mapping $g$ defined by:
    \begin{equation}
        \label{indicvariog2}
        g(x,y) = \frac{1}{2\pi} \arccos \rho(x,y), \quad x, y \in \mathbb{X},
    \end{equation}
    is the median indicator variogram of a Gaussian random field on $\mathbb{X}$ with correlation function $\rho$.
\end{proposition}

\subsection{Madograms}

Let $Z$ be a random field on $\mathbb{X}$. Its madogram, or first-order variogram, is defined as:
\begin{equation*}
    \gamma^{(1)}(x,y) = \frac{1}{2} \mathbb{E} \{ \lvert Z(x) - Z(y) \rvert \}, \quad x, y \in \mathbb{X},
\end{equation*}
provided that the expected value exists. Owing to Cauchy-Schwarz inequality, this is the case as soon as the variogram of $Z$ exists.

To date, little is known about the characterization of madograms, most of which are due to Matheron \cite{Matheron1989}. His results were established for a Euclidean space, but can be extended in a straightforward manner to any set of points. In particular:
\begin{itemize}
    \item The negative-type and triangle inequalities (Eqs. \ref{condneg} and \ref{triang}) hold for a madogram.

    \item Any madogram can be expressed as an integral mixture of indicator variograms.

    \item Any madogram is the limit of a sequence of functions proportional to indicator variograms. 
\end{itemize}

An interesting subclass of madograms is that of first-order stationary Gaussian random fields: 
\begin{proposition}[Matheron \citep{Matheron1989}]
\label{prop02}
    Let $\mathbb{X}$ be a set of points and $\gamma$ a variogram on $\mathbb{X} \times \mathbb{X}$. Then, the mapping $\gamma^{(1)}$ defined by:
    \begin{equation}
        \label{madog}
        \gamma^{(1)}(x,y) = \sqrt{\frac{\gamma(x,y)}{\pi}}, \quad x, y \in \mathbb{X},
    \end{equation}
    is the madogram of a first-order stationary Gaussian random field with variogram $\gamma$. 
\end{proposition}

\section{Main Results}
\label{sec3}

Throughout this section, $\mathbb{X}$ denotes a set of points, not necessarily provided with a metric. We begin with a definition from Matheron \citep{matheron1976}. 

\begin{definition}
    A bivariate distribution is \emph{Hermitian} if it is a mixture of bivariate Gaussian distributions with standard Gaussian marginals. Such a bivariate distribution has an isofactorial representation involving the normalized Hermite polynomials. 
\end{definition}

By extension, we will say that a random field is \emph{Hermitian} if all its bivariate distributions are Hermitian.

\subsection{Characterization of Indicator Variograms}

\begin{proposition}
\label{closedconvex}
The set of indicator variograms is a closed convex subset of the set of real-valued functions defined on $\mathbb{X} \times \mathbb{X}$ endowed with the topology of simple convergence. 
\end{proposition}

\begin{theorem}
\label{thm1}
    A mapping $g: \mathbb{X} \times \mathbb{X} \to [0,+\infty)$ is the variogram of a first-order stationary indicator random field on $\mathbb{X}$ if, and only if, it has the following representation:
    \begin{equation}
        \label{indicvariog}
        g(x,y) = \frac{1}{2\pi} \int_{0}^{+\infty} \arccos \rho_{\omega}(x,y) \, F({\rm d}\omega), \quad x,y \in \mathbb{X},
    \end{equation}
    where $F$ is a cumulative distribution function on $(0,+\infty)$ and, for all $\omega \in (0,+\infty)$, $\rho_{\omega}$ is a correlation function on $\mathbb{X} \times \mathbb{X}$. 
\end{theorem}

\begin{corollary}
\label{hull}
    The set of indicator variograms is the closed convex hull of the set of median indicator variograms of Gaussian random fields. 
\end{corollary}

\begin{remark}
    Martins de Carvalho and Clark \citep{Martins1983} showed that, in the Euclidean space, it is not true that the set of stationary indicator variograms is the closed convex hull of the set of stationary median indicator variograms of Gaussian random fields, i.e., when the mappings $g$ and $\rho_{\omega}$ in Eq. (\ref{indicvariog}) are functions of $x-y$ instead of $(x,y)$. However, the statement becomes true when lifting the assumption of stationarity. Even if the mapping $g$ is stationary, the representation (\ref{indicvariog}) may require the mappings $\rho_{\omega}$ to be non-stationary. 
\end{remark}

\begin{remark}
\label{rem2}
    An equivalent formulation of Theorem \ref{thm1} is the following: A mapping $g: \mathbb{X} \times \mathbb{X} \to [0,+\infty)$ is the variogram of a first-order stationary indicator random field on $\mathbb{X}$ if, and only if, it has the following representation:
    \begin{equation}
        g(x,y) = \frac{1}{\pi} \int_{0}^{+\infty} \arcsin \sqrt{\frac{\gamma_{\omega}(x,y)}{2}} \, F({\rm d}\omega), \quad x, y \in \mathbb{X},
    \end{equation}
    where $F$ is a cumulative distribution function on $(0,+\infty)$ and, for all $\omega \in (0,+\infty)$, $\gamma_{\omega}$ is the variogram of a standard Gaussian random field on $\mathbb{X}$. 
\end{remark}

\begin{corollary}[Existence theorem for indicator variograms]
\label{cor01}
    The variogram of a first-order stationary indicator random field is the median indicator variogram of a Hermitian random field. 
\end{corollary}

\begin{proposition}
\label{thm3}
    Let $g, g^\prime: \mathbb{X} \times \mathbb{X} \to [0,+\infty)$ be indicator variograms. Then:
    \begin{enumerate}
    \item For any $\varpi \in [0,1]$, $\varpi \, g$ is an indicator variogram.
    \item For any $\varpi \in [0,1]$, $\varpi \, g + (1-\varpi) \, g^\prime$ is an indicator variogram.
    \item $g+g'-4 \, g \, g^\prime$ is an indicator variogram. 
    \end{enumerate}
\end{proposition}

\begin{remark}
Owing to the first statement in Proposition \ref{thm3}, the cumulative distribution function $F$ in Theorem \ref{thm1} can be replaced by a non-decreasing function with values in $[0,\varpi]$ (with $0\leq \varpi \leq 1$) instead of $[0,1]$. 
\end{remark}

\begin{theorem}
\label{g2exp}
    If $g$ is an indicator variogram, so is $\frac{\varpi}{4}(1-\exp(-t g))$ for any $t > 0$ and any $\varpi \in [0,1]$. 
\end{theorem}

\begin{theorem}
\label{thm4}
    Let $g: \mathbb{X} \times \mathbb{X} \to [0,+\infty)$ be the variogram of a first-order stationary indicator random field on $\mathbb{X}$. Then, for any positive integer $n$ and any $x_1,\ldots,x_n \in \mathbb{X}$, it fulfills the following inequalities:
    \begin{enumerate}
        
        \item Negative-type inequalities (\ref{condneg}).

        \item Upper-bound inequality (\ref{upperbound}).

        \item Polygonal inequalities \citep{Deza1961}: 
        \begin{equation}
        \label{polyg}
           \sum_{k=1}^{\lfloor \frac{n}{2} \rfloor} \sum_{\ell=\lfloor \frac{n}{2} \rfloor+1}^n g(x_k,x_\ell) \geq \sum_{1 \leq k<m \leq \lfloor \frac{n}{2} \rfloor} g(x_k,x_m) + \sum_{\lfloor \frac{n}{2} \rfloor<\ell<m\leq n}  g(x_\ell,x_m).
        \end{equation}
        A particular case is the triangle inequality (\ref{triang}) when $n=3$.
        
        \item Odd-clique inequalities \citep{Barahona1986}: for all $\lambda_1,\ldots,\lambda_n \in \{-1,0,1\}$ such that $\sum_{k=1}^n \lambda_k$ is odd,
        \begin{equation}
        \label{oddcl}
           \sum_{k=1}^n \sum_{\ell=1}^n \lambda_k \, \lambda_\ell \, g(x_k,x_\ell) \leq 0.
        \end{equation}
        These inequalities include Matheron's inequalities (\ref{1clique}).

        \item Hypermetric inequalities \citep{Deza1961, Kelly1975}: for all $\lambda_1,\ldots,\lambda_n \in \mathbb{Z}$ such that $\sum_{k=1}^n \lambda_k = 1$,
        \begin{equation}
        \label{hyperm}
           \sum_{k=1}^n \sum_{\ell=1}^n \lambda_k \, \lambda_\ell \, g(x_k,x_\ell) \leq 0.
        \end{equation}

        \item Positive semidefinite inequalities \citep{Laurent1995}: for all $\lambda_1,\ldots,\lambda_n \in \mathbb{Z}$,
        \begin{equation}
        \label{psd}
           \sum_{k=1}^n \sum_{\ell=1}^n \lambda_k \, \lambda_\ell \, g(x_k,x_\ell) \leq \frac{1}{4} \sigma(\boldsymbol{\lambda})^2,
        \end{equation}
        where $\boldsymbol{\lambda} = (\lambda_1,\ldots,\lambda_n)$ and $\sigma(\boldsymbol{\lambda}) = \sum_{k=1}^n \lambda_k$.

        \item Rounded positive semidefinite inequalities \citep{Letchford2012}: for all $\lambda_1,\ldots,\lambda_n \in \mathbb{Z}$ such that $\sum_{k=1}^n \lambda_k$ is odd,
        \begin{equation}
        \label{roundedpsd}
           \sum_{k=1}^n \sum_{\ell=1}^n \lambda_k \, \lambda_\ell \, g(x_k,x_\ell) \leq \lfloor \frac{1}{4}  \sigma(\boldsymbol{\lambda})^2\rfloor.
        \end{equation}
        These inequalities include Shepp's inequalities (\ref{shepp63}).

         \item Gap inequalities \citep{Laurent1996}: for all $\lambda_1,\ldots,\lambda_n \in \mathbb{Z}$,
        \begin{equation}
        \label{gap}
           \sum_{k=1}^n \sum_{\ell=1}^n \lambda_k \, \lambda_\ell \, g(x_k,x_\ell) \leq \frac{1}{4} \left(\sigma(\boldsymbol{\lambda})^2 - \gamma(\boldsymbol{\lambda})^2 \right),
        \end{equation}
        where $\boldsymbol{\lambda} = (\lambda_1,\ldots,\lambda_n)$ and $\gamma(\boldsymbol{\lambda}) = \min \{ \lvert \boldsymbol{\lambda} \boldsymbol{z}^\top \rvert: \boldsymbol{z} \in \{-1,1\}^n \}$.
       
        \item Corner-positive inequalities (\ref{corner}). 

    \end{enumerate}

\end{theorem}
    
\begin{remark}
    The gap inequalities (\ref{gap}) are difficult to check in practice, as the calculation of the gap $\gamma(\boldsymbol{\lambda})$ is NP-hard \citep{Laurent1995}. The weaker rounded positive semidefinite inequalities, which imply inequalities (\ref{polyg}) to (\ref{psd}), are easier to check. 
\end{remark}

\subsection{Characterization of Madograms}

\begin{theorem}
\label{thm2}
    A mapping $\gamma^{(1)}: \mathbb{X} \times \mathbb{X} \to [0,+\infty)$ is the madogram of a random field on $\mathbb{X}$ if, and only if, it has the following representation:
    \begin{equation}
        \label{madogram}
        \gamma^{(1)}(x,y) = \int_{0}^{+\infty} \sqrt{\gamma_{\omega}(x,y)} \, F({\rm d}\omega),
    \end{equation}
    where $F$ is a cumulative distribution function on $(0,+\infty)$ and, for all $\omega \in (0,+\infty)$, $\gamma_{\omega}$ is a variogram on $\mathbb{X} \times \mathbb{X}$. 
\end{theorem}

\begin{corollary}
\label{cor2}
    Let $\gamma^{(1)}: \mathbb{X} \times \mathbb{X} \to [0,+\infty)$ be the variogram of a first-order stationary indicator random field on $\mathbb{X}$. Then it has the representation (\ref{madogram}).  
\end{corollary}

Note that the reciprocal of Corollary \ref{cor2} is not true: A mapping having the representation (\ref{madogram}) may not satisfy the upper-bound inequality (\ref{upperbound}) and, therefore, may not belong to the set of indicator variograms. However, an equivalence between madograms and indicator variograms can be established under an additional assumption of boundedness of the random field distribution, as stated in the next theorem. 

\begin{theorem}
    \label{mad2ind}
A mapping $\gamma^{(1)}: \mathbb{X} \times \mathbb{X} \to [0,\frac{1}{2}]$ is the madogram of a random field on $\mathbb{X}$ with values in $[0,1]$ and stationary marginal distribution if, and only if, $\gamma^{(1)}$ is the variogram of a first-order stationary indicator random field.     
\end{theorem}

\begin{corollary}
    \label{mad2ind2}
    The  madogram of a random field with bounded values and stationary marginal distribution is, up to a multiplicative factor, the variogram of a first-order stationary indicator random field. 
\end{corollary}

\begin{theorem}
\label{thm5}
    Let $\gamma^{(1)}: \mathbb{X} \times \mathbb{X} \to [0,+\infty)$ be the madogram of a random field on $\mathbb{X}$. Then, for any positive integer $n$ and any $x_1,\ldots,x_n \in \mathbb{X}$, it fulfills the following inequalities:
    \begin{enumerate}     
        \item negative-type inequalities (\ref{condneg});
        \item polygonal inequalities (\ref{polyg});
        \item odd-clique inequalities (\ref{oddcl});
        \item hypermetric inequalities (\ref{hyperm}). 
    \end{enumerate}
    If, furthermore, the marginal distribution of the random field is stationary and valued in $[0,1]$, then $\gamma^{(1)}$ also fulfills
    \begin{enumerate}
        \item the upper-bound inequality (\ref{upperbound});
        \item the positive semidefinite inequalities (\ref{psd});
        \item the rounded positive semidefinite inequalities (\ref{roundedpsd});
        \item the gap inequalities (\ref{gap});
        \item the corner-positive inequalities (\ref{corner}). 
    \end{enumerate}
\end{theorem}

\section{Examples of Indicator Variograms}
\label{sec:examples}

In the next three subsections, $\mathbb{X}$ is a metric space, i.e., a set of points endowed with a distance (see \ref{app_background1} for background details on distance geometry): 
\begin{itemize}
    \item the $N$-dimensional Euclidean space $\mathbb{R}^N$ endowed with the Euclidean distance $d_\text{E}$ (Section \ref{sec41}); 
    \item the $N$-dimensional unit sphere $\mathbb{S}^N(1)$ endowed with the geodesic (great circle) distance $d_{\text{GC}}$ (Section \ref{sec42}); 
    \item an undirected simple finite weighted graph endowed with the square-rooted resistance distance ($\sqrt{d_\text{R}}$) or an undirected simple finite unweighted graph endowed with the communicability distance ($d_\text{C}$) (Section \ref{sec43}).
\end{itemize}
In contrast, in Section \ref{sec44}, $\mathbb{X}$ is any abstract set of points, not necessarily endowed with a distance.

\subsection{Indicator Variograms in Euclidean Spaces}
\label{sec41}

\begin{example}
\label{example3}
    Let $\gamma$ be the radial part of the variogram of an isotropic standard Gaussian random field in $\mathbb{R}^N$, and let $\varpi \in [0,1]$. Then, the following mapping is a valid indicator variogram:
    \begin{equation}
    \label{seriesmodel}
        g(x,y) = \frac{2\varpi}{\pi^2} \sum_{k=0}^{+\infty} \frac{\gamma((2k+1) d_{\text{E}}(x,y))}{(2k+1)^2}, \quad x,y \in \mathbb{R}^N.
    \end{equation}
\end{example}

Particular cases are provided in Table \ref{tab:euclideanexamples}. These models depend on two parameters (a scale parameter $\lambda$ and a sill parameter $\varpi$), reach a maximum sill of $0.25$ and behave linearly near the origin (Figure \ref{fig:elementary}). The modified Bessel function of the first kind $I_{\frac{3}{2}}$ is given by \citep{prud2}
$$I_{\frac{3}{2}}(t) = \sqrt{\frac{2}{\pi}} \frac{t \cosh t - \sinh t}{t^{\frac{3}{2}}}, \quad t > 0,$$
so that all the models in Table \ref{tab:euclideanexamples} can be expressed in terms of elementary functions. 

\begin{table}[ht]
\footnotesize
\label{tab:euclideanexamples}
\begin{tabular}{ l  l  l}
\hline
Variogram model & Expression of $g(x,y)$ & Restrictions \\
\hline
Hyperbolic tangent 1 & $\frac{\varpi \, d_\text{E}(x,y)}{4 \lambda} \tanh \left(\frac{\lambda}{d_\text{E}(x,y)} \right)$ & $\lambda>0$, $\varpi \in [0,1]$, $N>0$ \\
Hyperbolic tangent 2 & $\frac{\varpi}{8} \left[\frac{3 \, d_\text{E}(x,y)}{\lambda} \tanh\left(\frac{\lambda}{d_\text{E}(x,y)}\right) + \tanh^2\left(\frac{\lambda}{d_\text{E}(x,y)}\right)-1 \right]$ & $\lambda>0$, $\varpi \in [0,1]$, $N>0$ \\
I-Bessel & $\frac{3 \varpi \sqrt{\pi d_\text{E}(x,y)}}{2\sqrt{\lambda}\left[1-\exp\left(-\frac{\lambda}{d_\text{E}(x,y)}\right)\right]} \exp\left(-\frac{\lambda}{2d_\text{E}(x,y)}\right) I_{\frac{3}{2}}\left(\frac{\lambda}{2d_\text{E}(x,y)}\right)$ & $\lambda>0$, $\varpi \in [0,1]$, $N=1$ \\
\hline
\end{tabular}
\caption{Elementary indicator variograms in $\mathbb{R}^N \times \mathbb{R}^N$}
\end{table}

\begin{figure}[h]
    \centering
\includegraphics[width = 0.55\textwidth]{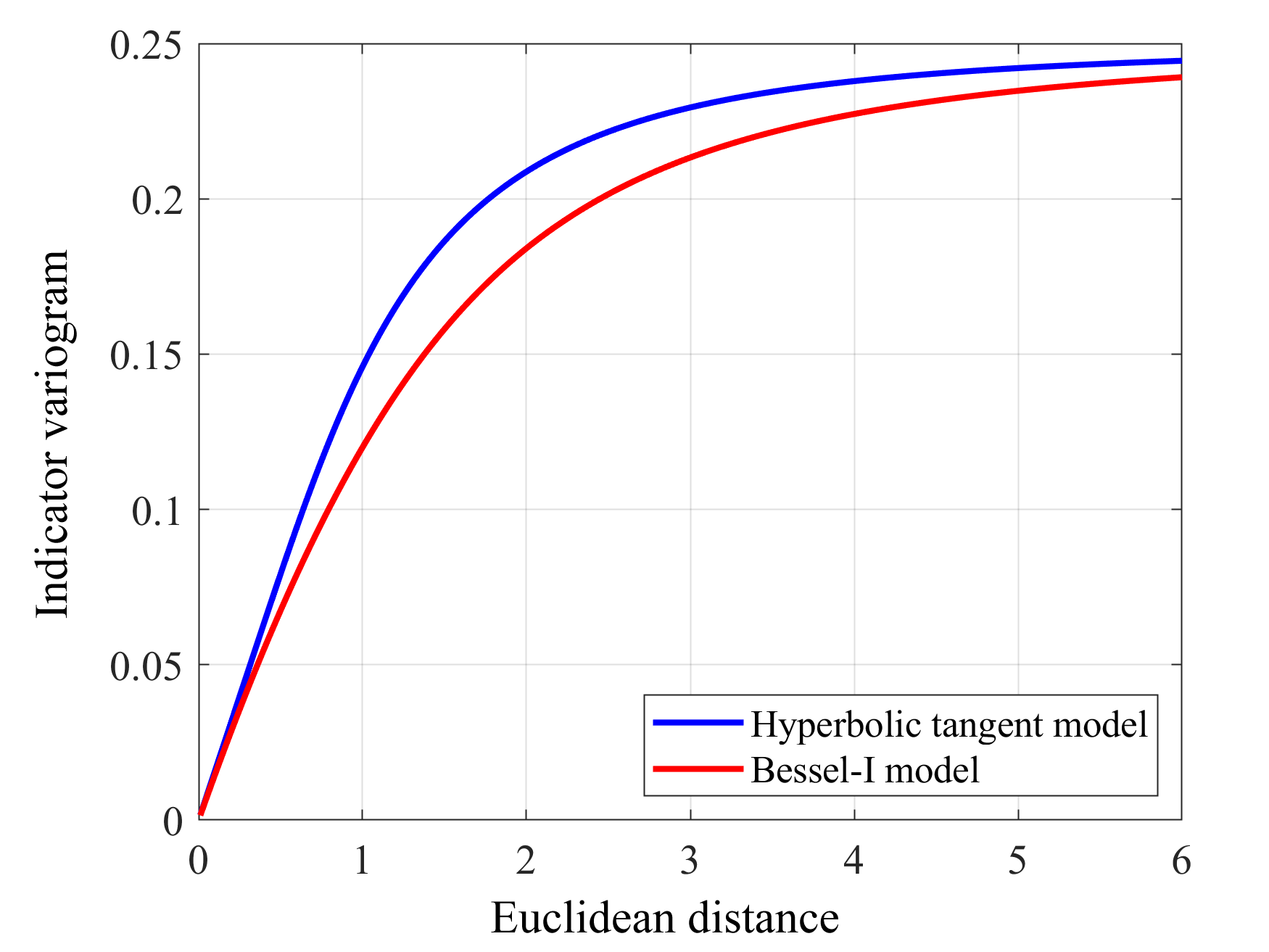}
    \caption{Elementary models that are admissible as indicator variograms in Euclidean spaces: hyperbolic tangent model 1 with $\lambda = \frac{\pi}{2}$ and $\varpi=1$, and I-Bessel model with $\lambda = 10$ and $\varpi=1$.}
    \label{fig:elementary}
\end{figure}

\begin{example}
\label{madex}
Let $\gamma$ be the variogram of a stationary standard Gaussian random field in $\mathbb{R}^N$, and let $\varpi \in [0,1]$. Then, the mapping $g$ defined by
    \begin{equation}
    \label{seriesmodel2}
        g(x,y) = \frac{4\varpi}{\pi^2} \sum_{k=1}^{+\infty} \frac{\gamma(2k(x-y))}{4k^2-1}, \quad x,y \in \mathbb{R}^N,
    \end{equation}
is a valid indicator variogram. 
\end{example}

\begin{example}[Completely monotone covariances]
\label{complmonot}
Let $\varphi: [0,+\infty) \to (0,\frac{1}{4}]$ be completely monotone. Then, the mapping $g$ defined by
\begin{equation}
\label{completmonot}
    g(x,y) = \varphi(0)-\varphi(d_\text{E}(x,y)), \quad x, y \in \mathbb{R}^N,
\end{equation}
is a valid indicator variogram.     
\end{example}

Recall that, owing to Bernstein's theorem, a completely monotone function can be written as a scale mixture of exponential covariances: \begin{equation}
    \label{berns}
    \varphi(x) = \varphi(0) \int_0^{+\infty} \exp(-t x) {\rm d}F(t),
\end{equation} 
where $F$ is the cumulative distribution function of a probability measure on $[0,+\infty)$.

Particular cases of indicator variograms of the form (\ref{completmonot}) \citep{emery2006} are provided in Table \ref{tab:complmonot}. These models depend on two to three parameters (a scale parameter $a$, a sill parameter $\varpi$ and, possibly, a shape parameter $b$) and reach a maximum sill of $0.25$. 

\begin{table}[ht]
\footnotesize
\label{tab:complmonot}
\begin{tabular}{ l  l  l}
\hline
Variogram model & Expression of $g(x,y)$ & Restrictions \\
\hline
Exponential & $\frac{\varpi}{4}-\frac{\varpi}{4}\exp(-a \, d_\text{E}(x,y))$ & $a>0$, $\varpi \in [0,1]$, $N>0$ \\
Gamma & $\frac{\varpi}{4}-\frac{\varpi}{4}\left(1+\frac{d_\text{E}(x,y)}{a}\right)^{-b}$ & $a>0$, $b>0$, $\varpi \in [0,1]$, $N>0$ \\
Stable & $\frac{\varpi}{4}-\frac{\varpi}{4}\exp(-a \, d^b_\text{E}(x,y))$ & $a>0$, $b \in (0,1]$, $\varpi \in [0,1]$, $N>0$ \\
Mat\'ern & $\frac{\varpi}{4} - \frac{\varpi}{2^{b+1} \, \Gamma(b)} \left(\frac{d_\text{E}(x,y)}{a}\right)^b {\cal K}_b\left(\frac{d_\text{E}(x,y)}{a}\right)$ & $a>0$, $b \in (0,\frac{1}{2}]$, $\varpi \in [0,1]$, $N>0$ \\
\hline
\end{tabular}
\caption{Indicator variograms, expressed in terms of elementary functions or of modified Bessel function of the second kind (${\cal K}$), that are valid in $\mathbb{R}^N \times \mathbb{R}^N$}
\end{table}

\subsection{Indicator Variograms on Spheres}
\label{sec42}

\begin{example}[Linear variogram on the $N$-sphere]
\label{linexample}
For any $\varpi \in [0,1]$, the mapping $g$ defined by
\begin{equation*}
    g(x,y) = \frac{\varpi\, d_{\text{GC}}(x,y)}{2 \pi}, \quad x, y \in \mathbb{S}^N(1),
\end{equation*}
is a valid indicator variogram. 
\end{example}

\begin{example}[Exponential variogram on the $N$-sphere]
\label{expexample}
For any $\varpi \in [0,1]$ and $t>0$, the mapping $g$ defined by
\begin{equation*}
    g(x,y) = \frac{\varpi}{4} \left(1- \exp(- t \, d_{\text{GC}}(x,y))\right), \quad x, y \in \mathbb{S}^N(1),
\end{equation*}
is a valid indicator variogram. 
\end{example}

\begin{example}[Triangular wave on the circle]
\label{triexample}
For any $\varpi \in [0,1]$ and any positive integer $k$, the mapping $g$ defined by
\begin{equation}
\label{triangwav}
    g(x,y) = \frac{\varpi}{2\pi} \min_{n \in \mathbb{N}} \, \lvert k \, d_{\text{GC}}(x,y) - 2n \pi \rvert, \quad x, y \in \mathbb{S}^1(1),
\end{equation}
is a valid indicator variogram.     
\end{example}

\begin{example}[Quadratic variogram on the circle]
\label{quadexample}
For any $\varpi \in [0,1]$, the mapping $g$ defined by
\begin{equation*}
    g(x,y) = \frac{3\varpi}{8\pi^2} d_{\text{GC}}(x,y) (2\pi-d_{\text{GC}}(x,y)), \quad x, y \in \mathbb{S}^1(1),
\end{equation*}
is a valid indicator variogram.     
\end{example}

\subsection{Indicator Variograms on Undirected Simple Finite Graphs}
\label{sec43}

\begin{example}
\label{examplegraph1}
    The hyperbolic tangent models in Table \ref{tab:euclideanexamples}, as well as the models of the form (\ref{completmonot}) are valid indicator variograms on graphs, when substituting the Euclidean distance $d_\text{E}$ with the communicability distance $d_\text{C}$ or with the square-rooted resistance distance $\sqrt{d_\text{R}}$. 
\end{example}

\begin{remark}
    The models in Example \ref{examplegraph1} are also valid for the so-called \emph{graphs with Euclidean edges} endowed with the resistance distance. Such graphs, introduced by Anderes et al. \citep{Anderes2020}, contain not only the vertices, but also the edges.
\end{remark}

\subsection{Indicator Variograms in Any Abstract Set of Points}
\label{sec44}

\begin{example}[Variogram models with a partial nugget effect]
    \label{nugget}
    Let $\varpi \in [0,1]$ and let $C$ be the covariance of a random field in $\mathbb{X}$ with values in $[-1,1]$. Then, the following mapping is a valid indicator variogram on $\mathbb{X} \times \mathbb{X}$:
    \begin{equation*}
        g(x,y) = \begin{cases}
            0  \text{ if $x=y$}\\
            \frac{\varpi}{4} (1 - C(x,y)) \text{ otherwise.}
        \end{cases}
    \end{equation*}    
\end{example}

As particular cases, let us mention: 
\begin{itemize}
    \item If $\mathbb{X} = \mathbb{R}^N$, $C(x,y) = \frac{1}{2} \rho(x-y)$, where $\rho$ is a continuous stationary correlation function \citep{Shamai1991}. 

    \item If $\mathbb{X} = \mathbb{S}^1$, $C(x,y) = \frac{1}{2} \rho(d_{\text{GC}}(x,y))$, where $\rho$ is a continuous isotropic correlation function. A random field on the circle with values in $[-1,1]$ and covariance $C$ can be constructed with a cosine wave \citep{alegria2020}. 

    \item If $\mathbb{X} = \mathbb{R}^3$, $C(x,y) = \frac{18}{35} cub_a(x,y)$, where $cub_a$ is the cubic correlation with range $a>0$: $$cub_a(x,y) = \begin{cases}
        0 \text{ if $d_{\text{E}}(x,y) \geq a$}\\
        1-7 \left(\frac{d_{\text{E}}(x,y)}{a}\right)^2 + \frac{35}{4} \left(\frac{d_{\text{E}}(x,y)}{a}\right)^3 - \frac{7}{2} \left(\frac{d_{\text{E}}(x,y)}{a}\right)^5+\frac{3}{4} \left(\frac{d_{\text{E}}(x,y)}{a}\right)^7 \text{ otherwise}.
        \end{cases}$$ A random field in $\mathbb{R}^3$ with values in $[-1,1]$ and covariance $C$ can be constructed with a partition method \citep{emery2006}. 
\end{itemize}

\begin{example}[Special functions]
\label{generalexamples}
Variogram models involving the error function (erf) and hypergeometric function ${}_pF_q$, with $p, q \in \mathbb{N}$, are presented in Table \ref{tab:generalexamples}. Many closed-form expressions of the hypergeometric functions are available in the literature \citep{prud3, weisstein}. 
\end{example}

\begin{landscape}
\begin{table}[ht]
\small
\label{tab:generalexamples}
\begin{tabular}{l l  l}
\hline
Model & Expression of $g(x,y)$ & Restrictions \\
\hline
1& $\frac{\varpi}{4} - \frac{\varpi}{4}\exp(k a \gamma(x,y)) \left[ 1-\text{erf}(\sqrt{a \gamma(x,y)})\right]^k$ & $a>0$, $k \in \mathbb{N}^*$, $\varpi \in [0,1]$ \\
2& $\frac{\varpi}{4} -  \frac{\varpi}{4}\left[\text{erf}\left(\frac{1}{\sqrt{a \gamma(x,y)}} \right) - \frac{1}{\pi} \sqrt{a \gamma(x,y)} \left(1-\exp\left(-\frac{1}{a \gamma(x,y)}\right)\right)\right]^k$ & $a>0$, $k \in \mathbb{N}^*$, $\varpi \in [0,1]$ \\
3& $\sqrt{\frac{\gamma(x,y)}{2\pi(1-\beta)}} \frac{\Gamma(\lambda)}{2\Gamma(\frac{1}{2}+\lambda)} {}_2F_1\left(\frac{1}{2},\frac{1}{2};\frac{1}{2}+\lambda;\frac{\gamma(x,y)-2\beta}{2-2\beta}\right)$ & $\beta \in (-1,1)$, $\lambda \geq \max\{1,\frac{1}{1-\beta}\}$ \\
4& $\sqrt{\frac{\gamma(x,y)}{2\pi}} \, \frac{\Gamma(\lambda) \Gamma(\frac{1}{2}+\lambda-\alpha-\beta)}{2\Gamma(\frac{1}{2}+\lambda-\alpha) \Gamma(\frac{1}{2}+\lambda-\beta)} {}_3F_2\left(\frac{1}{2},\frac{1}{2}+\lambda-\alpha-\beta,\lambda;\frac{1}{2}+\lambda-\alpha,\frac{1}{2}+\lambda-\beta;\frac{\gamma(x,y)}{\gamma(x,y)-2}\right)$ & $\alpha \in (\frac{1}{2},\frac{3}{2})$, $\beta \in (\frac{1}{2}-\alpha,\frac{1}{2})$, $\lambda \geq \alpha+\beta+\frac{1}{2}$ \\
5& $\frac{1}{\pi \lambda} \sqrt{\frac{\gamma(x,y)}{2}} \, {}_{3}F_2\left(\frac{1}{2},\frac{1}{2},1;\frac{1+\lambda}{2},1+\frac{\lambda}{2};{\frac{\gamma(x,y)}{2}}\right)$ & $\lambda > 1$ \\
6& $\frac{\Gamma(\alpha+\beta) \Gamma(\alpha+\frac{1}{2})}{\pi \Gamma(\alpha) \Gamma(\alpha+\beta+\frac{1}{2})} \sqrt{\frac{\gamma(x,y)}{2}} \, {}_3F_2\left(\frac{1}{2},\frac{1}{2},\alpha+\frac{1}{2};\frac{3}{2},\alpha+\beta+\frac{1}{2} ;\frac{\gamma(x,y)}{2}\right)$ & $\alpha > 0$, $\beta>0$ \\
\hline
\end{tabular}
\caption{Indicator variograms on $\mathbb{X} \times \mathbb{X}$, where $\mathbb{X}$ is a set of points. In all cases, $\gamma$ is the variogram of a standard Gaussian random field on $\mathbb{X}$}
\end{table}
\end{landscape}

\begin{remark}
    Models 3 to 6 in Table \ref{tab:generalexamples} with $\lambda=1$ and $\beta=0$ lead to the model in (\ref{indicvariog2}) with $\rho = 1-\gamma$, owing to the hypergeometric representation of the arccosine function, see formulae 1.623.1, 1.626.2, 9.121.26 and 9.131.1 of Gradshteyn and Ryzhik \citep{Grad}.
\end{remark}

\section{Discussion}
\label{sec:discussion}

\subsection{The Particular Role of Median Indicators}

Theorem \ref{thm1} and Corollary \ref{cor01} are disturbing, as they state that any indicator variogram is a median indicator variogram, in circumstances where the variance of a median indicator ($0.25$) is greater than the variance of any indicator whose mean value differs from $0.5$. Such a paradox can be explained because the variogram of a random field does not necessarily reach the value of its variance, even when considering two very distant points in $\mathbb{X}$.

For instance, for a stationary random field in a Euclidean space, the variance is equal to the variogram sill under an assumption of \emph{ergodicity}. However, if in Eq. (\ref{indicvariog}) one considers correlation functions $\rho_{\omega}(x,y)$ that do not tend to zero as $d_{\text{E}}(x,y)$ tends to infinity, then the variogram $g(x,y)$ will not tend to $0.25$ and can therefore model an indicator random field whose mean value is different from $0.5$.

The following construction confirms the previous statements. Starting from any indicator random field $Z$, one can define an indicator random field $\widetilde{Z}$ with mean $0.5$ and with the same variogram as $Z$ by posing
\begin{equation*}
         \widetilde{Z} = \begin{cases}
            Z \text{ with probability $\frac{1}{2}$}\\
            1-Z \text{ with probability $\frac{1}{2}$.}            
        \end{cases}           
\end{equation*}

\subsection{Model Internal Consistency}

What are the implications of modeling an indicator random field with an invalid variogram? One consequence, mentioned in the introductory section, is that one can obtain negative probabilities for certain events on the random set associated with the indicator random field. 

For applications related to spatial prediction via indicator kriging, the above problem can be seen as that of a misspecification of the kriging parameters. Although this problem has been mostly studied for random fields in Euclidean spaces (see Section 3.4.4 of Chil\`es and Delfiner \citep{chiles_delfiner_2012} and references therein) or when the number of data points becomes infinitely large \citep{Stein1988, cressie1993}, some general conclusions can be drawn:
\begin{enumerate}
    \item The impact of a variogram misspecification in the prediction is small provided that the behavior of the variogram at small distances is well modeled, while a considerable impact can be observed otherwise. That is, what matters for spatial prediction purposes is the modeling of the indicator variogram $g(x,y)$ for $x$ close to $y$.
    \item The impact of a variogram misspecification in the prediction error variance is more sensitive to a misspecification of the variogram: not only the behavior at small distances matters, but also its shape, range of values, and behavior at large distances. In any case, although it may not reflect the true prediction error, the calculated error variance will be non-negative if the modeled indicator variogram is a symmetric and conditionally negative semidefinite function. 
\end{enumerate}

Things are different for the reconstruction or realizability problem, i.e., for simulating an indicator random field with a specified variogram. Obviously, if this variogram is invalid, such an indicator random field does not exist. However, several variogram-based algorithms can still be used to construct indicator realizations, such as the sequential indicator algorithm \citep{Alabert1987} or optimization algorithms based on simulated annealing \citep{Torquato2002}. Rather than a genuine random field, these algorithms generate realizations of binary random vectors whose two-point set distributions are \emph{as close as possible} to target distributions represented by a (possibly invalid) indicator variogram. 

Even so, an exact reproduction of the target distributions is generally out of reach. For instance, sequential indicator simulation does not exactly reproduce the target indicator variogram---even if the latter is a valid model---unless this variogram leads to indicator kriging predictions that always lie in the interval $[0,1]$ \citep{Emery2004, chiles_delfiner_2012}, a condition that is violated by most variogram models. An illustration is given in Figure \ref{fig:sisim}: since it does not fulfills the triangle inequality, the cubic variogram is not realizable and the realizations exhibit a completely different variogram (linear near the origin instead of parabolic). In contrast, the exponential variogram is realizable; nonetheless, it is imperfectly reproduced by the realizations at small distances. Whether the spherical variogram is realizable for an indicator random field in $\mathbb{R}^2$ is still unknown, and the results show a poorer reproduction of the small-scale behavior than for the exponential variogram.

As a conclusion, variogram-based simulation algorithms do not provide any insight into the internal model consistency: they do not guarantee that the indicator realizations accurately reproduce the input variogram, not that this variogram is a valid indicator variogram. To correctly address the realizability problem, one must consider the hitting functional or all the finite-dimensional distributions of the indicator random field, which goes well beyond its variogram \citep{emeryortiz2011}.

\begin{figure}[h]
    \centering
\includegraphics[width = 0.995\textwidth]{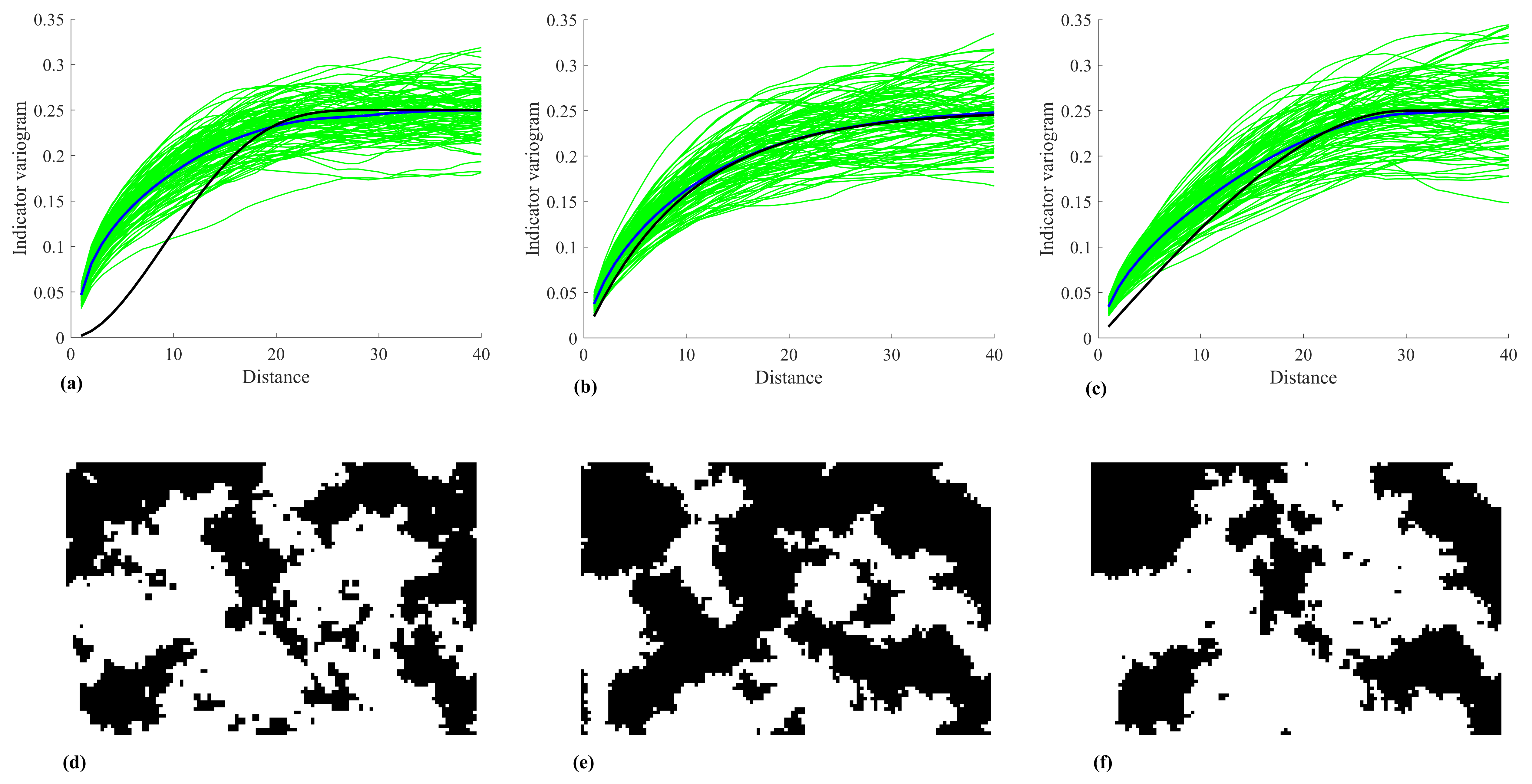}
    \caption{Top: Experimental variograms (green lines) and average experimental variogram (blue line) of $100$ realizations of an indicator with mean value $0.5$ simulated on a regular grid with $120 \times 80$ points in the 2D plane. Experimental variograms are calculated along the first axis. The theoretical variogram inputted in the sequential simulation algorithm is indicated in black and consists of an isotropic cubic (a), exponential (b) or spherical (c) model. Bottom: Examples of indicator realizations for the cubic (d), exponential (e) and spherical (f) models.}
    \label{fig:sisim}
\end{figure}

\subsection{Design of Simulation Algorithms}

As a by-product of our work, algorithms to simulate indicator random fields can be designed based on the characterizations and properties of the indicator variograms.  

\textbf{A general algorithm.} The representation (\ref{indicvariog}) makes it possible to design a general simulation algorithm:
\begin{enumerate}
    \item Simulate a random variable $T$ in $(0,+\infty)$ with cumulative distribution function $F$.
    \item Simulate a zero-mean Gaussian random field $Y_T$ with covariance function $\rho_T$. This can be done by the Cholesky decomposition approach or, if the number of target locations is large, by Gibbs sampling \citep{Arroyo2012}.
    \item Deliver the median indicator of $Y_T$. 
\end{enumerate}

For example, this algorithm can be used to simulate indicator random fields with completely monotone covariances in the Euclidean space (Example \ref{complmonot}), as these covariances are mixtures of exponential covariances, see Eq. (\ref{berns}).

\textbf{Ad hoc algorithms.} 
The above algorithm is not applicable if the cumulative distribution function $F$ and correlation functions $\rho_{\omega}$ in the representation (\ref{indicvariog}) are not known. In specific cases, however, it is possible to design ad hoc simulation algorithms based on the properties of the target indicator variogram.

For instance, consider the exponential variogram of Example \ref{expexample} with $\varpi=1$. Based on the proof of Theorem \ref{g2exp}, an indicator random field with this variogram can be simulated at any set of points $x_1,\ldots,x_n$ on the unit $N$-sphere in the following fashion:
\begin{enumerate}
    \item Simulate a Poisson random variable $K$ with mean $\frac{\pi}{2}t$.
    \item If $K=0$, assign the same value ($0$ or $1$, each with probability $0.5$) to all the target points.
    \item If $K>0$, generate $K$ independent copies of a random field $\widetilde{Z}$ with the linear variogram $g: (x,y) \mapsto \frac{2}{\pi} d_{\text{GC}}(x,y)$ and Rademacher marginal distribution. Denote $\widetilde{Z}_K$ the product of these $K$ copies and deliver $(1+\widetilde{Z}_K)/2$ as the simulated indicator random field.
\end{enumerate}

According to Lemma \ref{lem1}, the random field $\widetilde{Z}$ can be simulated by generating a zero-mean Gaussian random vector $(Y_1,\ldots,Y_n)$ with variance-covariance matrix equal to the Gram matrix of $x_1,\ldots,x_n$ (viewed as points of $\mathbb{R}^{N+1}$) and posing $\widetilde{Z}(x_i) = \text{sign}(Y_i)$ for $i = 1,\ldots,n$. In turn, the simulation of $(Y_1,\ldots,Y_n)$ can be done by means of a central limit approximation, by posing:
\begin{equation*}
    Y_i = \sqrt{\frac{N+1}{Q}} \sum_{q=1}^Q \epsilon_q \, x_{i,L_q},
\end{equation*}
where $x_{i,\ell}$ is the $\ell$-th Euclidean coordinate of $x_i$, $Q$ is a large positive integer, $L_1,\ldots,L_q$ are independent random variables uniformly distributed in $\{1,\ldots,N+1\}$, and $\epsilon_1,\ldots,\epsilon_q$ are independent random variables with a Rademacher distribution. Examples of realizations are depicted in Figure \ref{fig:expsphere}.

\begin{figure}[h]
    \centering
\includegraphics[width = 0.995\textwidth]{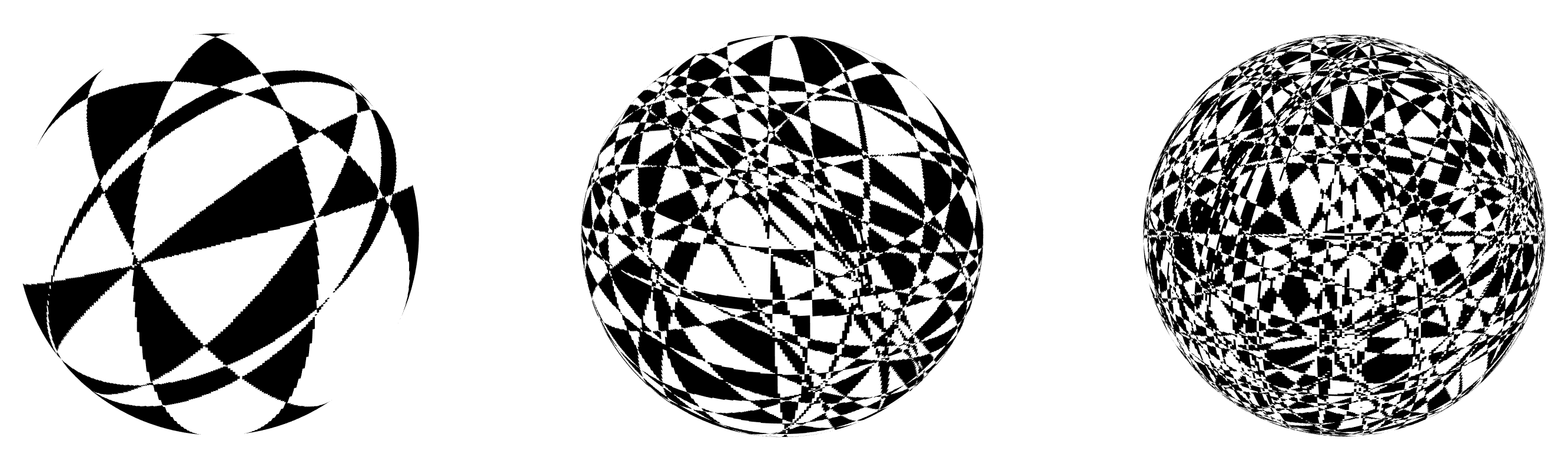}
    \caption{Realizations of indicator random fields on the $2$-sphere with mean value $0.5$ and exponential variogram $g(x,y)=\frac{1}{4}(1-\exp(-t\, d_\text{GC}(x,y))$. From left to right: $t=10$, $t=30$ and $t=50$.}
    \label{fig:expsphere}
\end{figure}

\section{Conclusions}
\label{sec:conclusions}
Necessary and sufficient conditions (NSCs) have been developed to characterize indicator variograms and madograms defined on any set of points. They may be used to design new theoretical models for these structural tools, or to check whether a given function is a valid model. To the best of the authors' knowledge, this is the first time that NSCs have been obtained for madograms. Regarding indicator variograms, the NSCs rest on the properties of Hermitian random fields, which likely makes them more functional than the untractable corner-positive inequalities designed by McMillan \citep{McMillan1955}.

The previous statements do not mean that everything has been sorted out. At the present time, closed-form expressions of many indicator variograms and madograms are still unknown, 
even for the most standard random field models, such as the indicator variograms of truncated Gaussian random fields, for which only integral or series expressions are available (\ref{excursion}). Finally, some valid variogram models (e.g., the well-known spherical variogram in a one-dimensional Euclidean space) may not be easily associated with the representations induced by the NSCs. This simply means that the topic addressed in this paper is a wide field of investigation.

\appendix

\section{Background}
\label{app_background}
\vspace*{12pt}

\subsection{Dissimilarities, Semi-distances and Distances}
\label{app_background1}

Let $\mathbb{X}$ be a set of points. The following definitions are classical in mathematical analysis and distance geometry \citep{Critchley1997}. 

\begin{definition}[Dissimilarity]
A mapping $d: \mathbb{X} \times \mathbb{X} \to [0,+\infty)$ is called a dissimilarity if it is symmetric and has zero diagonal elements, i.e., $d(x,y)=d(y,x)$ and $d(x,x)=0$ for any $x, y \in \mathbb{X}$. 
\end{definition}

\begin{definition}[Semi-distance]
A dissimilarity $d: \mathbb{X} \times \mathbb{X} \to [0,+\infty)$ is a semi-distance if it satisfies the triangle inequality, i.e., $d(x,y) + d(y,z) \geq d(x,z)$ for any $x, y, z \in \mathbb{X}$. 
\end{definition}

\begin{definition}[Distance]
A semi-distance $d: \mathbb{X} \times \mathbb{X} \to [0,+\infty)$ is a distance if it is definite, i.e., $d(x,y)=0 \implies x=y$ for any $x, y \in \mathbb{X}$. 
\end{definition}

A set $\mathbb{X}$ endowed with a distance is called a \emph{metric space}. Examples of distances and metric spaces include:
\begin{enumerate}
    \item The Euclidean distance in the $N$-dimensional Euclidean space $\mathbb{R}^N$: 
    \begin{equation}
    \label{euclidean}
        d_{\text{E}}(x,y) = \sqrt{\sum_{k=1}^N (x_k - y_k)^2}, \quad x=(x_1,\ldots,x_N), y = (y_1,\ldots,y_N).
    \end{equation}
    \item The geodesic (aka great-circle) distance on the $N$-dimensional sphere of radius $r>0$ (denoted $\mathbb{S}^N(r)$) embedded in $\mathbb{R}^{N+1}$: 
    \begin{equation}
    \label{geodesic}
        d_{\text{GC}}(x,y) = r \, \arccos\left(\frac{x^\top y}{r^2}\right), \quad x, y \in \mathbb{S}^N(r),
    \end{equation}
    with $\top$ standing for the transpose operator.
    \item The geodesic (aka shortest path) distance on an undirected simple finite weighted graph $\mathbb{G}$, defined as a finite set of vertices $x_1,\ldots ,x_n$ 
    such that (1) a pair of vertices $(x_k,x_\ell)$ can be connected by at most one edge associated with a positive weight $w_{k\ell}$, (2) the weight matrix $[w_{k\ell}]_{k,\ell=1}^n$ is symmetric, and (3) no edge connects a vertex with itself (no self-loop): 
    \begin{equation}
    \label{geodesic2}
        \begin{split}
        &d_{\text{SP}}(x,y) \\
        &= \begin{cases}
        +\infty \text{ if no path connects $x$ and $y$}\\
        \text{minimum sum of weights across all paths connecting $x$ and $y$ otherwise}.
        \end{cases}
        \end{split}
    \end{equation}
    \item The resistance distance $d_{\text{R}}$ on an undirected simple finite weighted graph $\mathbb{G}$ \citep{Klein1993}. 
    \item The communicability distance $d_{\text{C}}$ on an undirected simple finite unweighted graph $\mathbb{G}$ \citep{Estrada2012}.      
\end{enumerate}

\subsection{Random Fields, First-order Stationarity, and Separability}

Let $\mathbb{X}$ denote a set of points and $(\Omega,\mathcal{A},P)$ a probability space consisting of a sample space $\Omega$, an event space $\mathcal{A}$ and a probability function $P$. While most geostatistics textbooks only define random fields by a set of random variables indexed with the set $\mathbb{X}$, as was also done in the body of the paper, in the following appendices we provide a more technical definition that involves both $\mathbb{X}$ and $\Omega$ and will be essential for proving our findings.

A random field $Z$ on $\mathbb{X}$ is a mapping on $\mathbb{X} \times \Omega$, i.e., $Z = \{Z(x,\omega): x \in \mathbb{X}, \omega \in \Omega\}$, such that
\begin{enumerate}
    \item for any $x \in \mathbb{X}$, the mapping $Z(x,\cdot)$ is a random variable on $(\Omega,\mathcal{A},P)$;
    \item for any $\omega \in \Omega$, the mapping $Z(\cdot,\omega)$ is a function on $\mathbb{X}$ called a \emph{realization} of $Z$. 
\end{enumerate}

Particular cases of interest are:
\begin{itemize}
    \item \textit{Gaussian random fields:} $Z$ is a Gaussian random field if $\sum_{k=1}^n \lambda_k \, Z(x_k,\cdot)$ is a Gaussian random variable for any positive integer $n$ and any choice of $x_1, \ldots, x_n$ in $\mathbb{X}$ and $\lambda_1, \ldots, \lambda_n$ in $\mathbb{R}$.
    \item \textit{Indicator random fields:} $Z$ is an indicator random field if $Z(x,\omega) \in \{0,1\}$ for all $x \in \mathbb{X}$ and all $\omega \in \Omega$. 
\end{itemize}

The proofs in Appendix \ref{app_proof} rely on two assumptions on the random field under study. The first one, which is needed for the characterization of indicator variograms, is that of first-order stationarity, which means that the mapping $x \mapsto \mathbb{E}\{Z(x,\cdot)\}$, where $\mathbb{E}$ denotes the mathematical expectation on the probability space $(\Omega,\mathcal{A},P)$, exists and is constant over $\mathbb{X}$. Note that a random field whose excursion sets are first-order stationary for all thresholds has a stationary marginal distribution.   

The second assumption is that of separability, which means that the probabilities involving any uncountable set of points are uniquely determined from probabilities on countable sets of points. This second assumption can always be taken for granted, insofar as, for any random field, there exists a separable random field with the same spatial distribution \citep{Doob1953}, in particular, with the same variogram or madogram. An implication of separability is that the sample space $\Omega$, which contains the states indexing the realizations of the random field, can be chosen as an interval of $\mathbb{R}$, say $(0,1)$ or $(0,+\infty)$ \citep{Rokhlin1952, Ito1984}.

\section{Complements on Variograms of Any Order}
\label{extensionorder}
\vspace*{12pt}

\subsection{Definition and Known Results}

A generalization of the variogram and the madogram is the variogram of order $\alpha$, with $\alpha>0$, introduced by Matheron \citep{Matheron1989}. It is defined as
\begin{equation*}
    \gamma^{(\alpha)}(x,y) = \frac{1}{2} \mathbb{E} \bigl\{ \vert Z(x) - Z (y) \vert^\alpha \bigr\}, \quad x,y \in \mathbb{X},
\end{equation*}
provided that the expected value exists. This tool has been used in geostatistics to get a robust estimator of the variogram \citep{cressie1980} or to validate a bivariate distribution model \citep{emery2005}. The cases $\alpha=1$ and $\alpha=0.5$ correspond to the \emph{madogram} and \emph{rodogram}, respectively. 

The following propositions derive from Matheron \citep{Matheron1989} and Emery \citep{emery2005}. 

\begin{proposition}
\label{alpha0}
    The mapping $\gamma^{(\alpha)}$ is symmetric and vanishes on the diagonal of $\mathbb{X} \times \mathbb{X}$:
    \begin{equation*}
    \gamma^{(\alpha)}(x,x) = 0, \quad x \in \mathbb{X},
    \end{equation*}
    \begin{equation*}
    \gamma^{(\alpha)}(x,y) = \gamma^{(\alpha)}(y,x), \quad x, y \in \mathbb{X}.
    \end{equation*}
    Furthermore, it fulfills the negative-type inequalities when $\alpha \leq 2$, and the triangle inequality when $\alpha \leq 1$. 
\end{proposition}

\begin{proposition}
\label{alpha}
    Let $\alpha \leq \beta$ be positive real numbers. Any variogram of order $\alpha$ is also a variogram of order $\beta$. 
\end{proposition}

\begin{proposition}
\label{prop02b}
    Let $\mathbb{X}$ be a set of points, $\alpha$ a positive real number, and $\gamma$ a variogram on $\mathbb{X} \times \mathbb{X}$. Then, the mapping $\gamma^{(\alpha)}$ defined by:
    \begin{equation}
        \label{rodog}
        \gamma^{(\alpha)}(x,y) = \frac{2^{\alpha-1}}{\sqrt{\pi}} \Gamma\left(\frac{\alpha+1}{2}\right) [\gamma(x,y)]^{\alpha/2}, \quad x, y \in \mathbb{X},
    \end{equation}
    is the variogram of order ${\alpha}$ of a first-order stationary Gaussian random field with variogram $\gamma$. 
\end{proposition}

\subsection{New Results}

The following theorems and corollaries extend the results obtained for madograms. Proofs can be found in Appendix \ref{app_proof}. 

\begin{theorem}
\label{thm2c}
    For $\alpha>0$, a mapping ${\gamma}^{(\alpha)}: \mathbb{X} \times \mathbb{X} \to [0,+\infty)$ is the variogram of order $\alpha$ of a random field on $\mathbb{X}$ if, and only if, it has the following representation:
    \begin{equation}
        \label{pseudovariogalpha0}
        {\gamma}^{(\alpha)}(x,y) = \int_{0}^{+\infty} [{\gamma}_{\omega}(x,y)]^{\alpha/2} \, F({\rm d}\omega),
    \end{equation}
    where $F$ is a cumulative distribution function on $(0,+\infty)$ and, for all $\omega \in (0,+\infty)$, ${\gamma}_{\omega}$ is a variogram on $\mathbb{X} \times \mathbb{X}$. 
\end{theorem}

\begin{theorem}
\label{orderineq}
For $\alpha \leq 1$, any variogram of order $\alpha$ fulfills the inequalities stated in Theorem \ref{thm5}. For $\alpha \geq 1$, it fulfills the following triangle inequality: 
\begin{equation*}
    [{\gamma}^{(\alpha)}(x_1,x_2)]^{1/\alpha} + [{\gamma}^{(\alpha)}(x_2,x_3)]^{1/\alpha} \geq [{\gamma}^{(\alpha)}(x_1,x_3)]^{1/\alpha}, \quad x_1, x_2, x_3 \in \mathbb{X}. 
\end{equation*} 
\end{theorem}

\begin{theorem}
    \label{order2ind}
For $\alpha \leq 1$, a mapping $\gamma^{(\alpha)}: \mathbb{X} \times \mathbb{X} \to [0,\frac{1}{2}]$ is the variogram of order $\alpha$ of a random field on $\mathbb{X}$ with values in $[0,1]$ and stationary marginal distribution if, and only if, $\gamma^{(\alpha)}$ is the variogram of a first-order stationary indicator random field.   
\end{theorem}

\begin{corollary}
\label{alpha2beta}
    The variogram of order $\alpha$ of a random field with bounded values and stationary marginal distribution is also a variogram of order $\beta$, whatever $\alpha$ and $\beta$ in $(0,1]$.  
\end{corollary}

\section{Complements and Results in the Multivariate Setting}
\label{extensionmultivariate}
\vspace*{12pt}

In this section, $p$ be a positive integer. Also, all operations involving matrix-valued mappings, such as multiplication, integration or square rooting, are understood as elementwise.

\subsection{Matrix-valued Pseudo-variograms}
\label{pseud}

The pseudo-variogram of a $p$-variate random field $\boldsymbol{Z} = (Z_1,\ldots,Z_p)$ defined on $\mathbb{X}$ is the matrix-valued mapping $\boldsymbol{\gamma} = [\gamma_{ij}]_{i,j=1}^p$ whose entries are 
\begin{equation*}
    \gamma_{ij}(x,y) = \frac{1}{2} \mathbb{E}\{ [Z_i(x) - Z_j(y)]^2 \}, \quad x, y \in \mathbb{X}, \quad i,j = 1,\ldots,p.
\end{equation*}

This definition follows that of Clark et al. \citep{Clark1989}, which uses an expected squared difference and not a variance in the definition of the cross-entries. Note that the pseudo-variogram is always well-defined for an indicator random field, which takes its values in $\{0,1\}^p$. 
For a standard Gaussian random field (with zero mean and unit variance), the pseudo-variogram is related to the matrix-valued covariance $\boldsymbol{\rho}$ through
\begin{equation}
\label{vartocov}
    \boldsymbol{\gamma}(x,y) = \boldsymbol{1} - \boldsymbol{\rho}(x,y), \quad x, y \in \mathbb{X},
\end{equation}
where $\boldsymbol{1}$ is the all-ones matrix and $\boldsymbol{\rho}(x,y) := [\mathbb{E}\{ Z_i(x) \, Z_j(y)\}]_{i,j=1}^p$.

A matrix-valued mapping $\boldsymbol{\gamma}= [\gamma_{ij}]_{i,j=1}^p$ on $\mathbb{X} \times \mathbb{X}$ is a pseudo-variogram if, and only if, the following conditions hold \citep{Dorr2023}:
\begin{itemize}
    \item $\gamma_{ii}(x,x) = 0$ for any $i \in \{1,\ldots,p\}$ and $x \in \mathbb{X}$;
    \item $\gamma_{ij}(x,y) = \gamma_{ji}(y,x)$ for any $i,j \in \{1,\ldots,p\}$ and and $x,y \in \mathbb{X}$;
    \item negative-type inequalities: for all $x_1,\ldots,x_n \in \mathbb{X}$ and real weights $\{\lambda_{ik}: i=1,\ldots,p; k=1,\ldots,n\}$ such that $\sum_{i=1}^p \sum_{k=1}^n \lambda_{ik} = 0$,
    \begin{equation}
    \label{negtype}
    \sum_{i=1}^p \sum_{j=1}^p \sum_{k=1}^n \sum_{\ell=1}^n \lambda_{i,k} \, \lambda_{j,\ell} \, \gamma_{ij}(x_k,x_\ell) \leq 0.
    \end{equation}
\end{itemize}

More generally, one can define the pseudo-variogram of order $\alpha$ of a $p$-variate random field as the matrix-valued mapping $\boldsymbol{\gamma}^{(\alpha)}$ with entries
\begin{equation*}
    \gamma^{(\alpha)}_{i,j}(x,y) = \frac{1}{2} \mathbb{E} \{ \lvert Z_i(x) - Z_j(y) \rvert^{\alpha} \}, \quad x, y \in \mathbb{X}, \quad i,j = 1,\ldots,p,
\end{equation*}
provided that the expected value exists. For $\alpha=1$, this yields the pseudo-madogram. 

\subsection{Characterization of Indicator Pseudo-variograms}

\begin{theorem}
\label{thm1b}
    A matrix-valued mapping $\boldsymbol{g}: \mathbb{X} \times \mathbb{X} \to [0,+\infty)^{p \times p}$ is the pseudo-variogram of a first-order stationary indicator vector random field on $\mathbb{X}$ if, and only if, it has one of the following equivalent representations:
    \begin{equation}
        \label{indiccrossvariog}
        \begin{split}
        \boldsymbol{g}(x,y) &= \frac{1}{2\pi} \int_{0}^{+\infty} \arccos \boldsymbol{\rho}_{\omega}(x,y) \, F({\rm d}\omega) \\
        &= \frac{1}{\pi} \int_{0}^{+\infty} \arcsin \sqrt{\frac{\boldsymbol{\gamma}_{\omega}(x,y)}{2}} \, F({\rm d}\omega), \quad x,y \in \mathbb{X},
        \end{split}
    \end{equation}
    where $F$ is a cumulative distribution function on $(0,+\infty)$ and, for all $\omega \in (0,+\infty)$, $\boldsymbol{\rho}_{\omega}$ is a matrix-valued correlation function on $\mathbb{X} \times \mathbb{X}$ and $\boldsymbol{\gamma}_{\omega}$ is the pseudo-variogram of a $p$-variate standard Gaussian random field on $\mathbb{X} \times \mathbb{X}$. 
\end{theorem}

\subsection{Characterization of Pseudo-variograms of Order $\alpha$}

\begin{theorem}
\label{thm2d}
    For $\alpha>0$, a matrix-valued mapping $\boldsymbol{\gamma}^{(\alpha)}: \mathbb{X} \times \mathbb{X} \to [0,+\infty)^{p \times p}$ is the pseudo-variogram of order $\alpha$ of a $p$-variate random field on $\mathbb{X}$ if, and only if, it has the following representation:
    \begin{equation}
        \label{pseudovariogalpha}
        \boldsymbol{\gamma}^{(\alpha)}(x,y) = \int_{0}^{+\infty} [\boldsymbol{\gamma}_{\omega}(x,y)]^{\alpha/2} \, F({\rm d}\omega),
    \end{equation}
    where $F$ is a cumulative distribution function on $(0,+\infty)$ and, for all $\omega \in (0,+\infty)$, $\boldsymbol{\gamma}_{\omega}$ is the pseudo-variogram of a $p$-variate random field on $\mathbb{X} \times \mathbb{X}$. 
\end{theorem}

For $\alpha=1$, this translates into the following theorem. 

\begin{theorem}
\label{thm2b}
    A matrix-valued mapping $\boldsymbol{\gamma}^{(1)}: \mathbb{X} \times \mathbb{X} \to [0,+\infty)^{p \times p}$ is the pseudo-madogram of a $p$-variate random field on $\mathbb{X}$ if, and only if, it has the following representation:
    \begin{equation}
        \label{pseudomadogram}
        \boldsymbol{\gamma}^{(1)}(x,y) = \int_{0}^{+\infty} \sqrt{\boldsymbol{\gamma}_{\omega}(x,y)} \, F({\rm d}\omega),
    \end{equation}
    where $F$ is a cumulative distribution function on $(0,+\infty)$ and, for all $\omega \in (0,+\infty)$, $\boldsymbol{\gamma}_{\omega}$ is the pseudo-variogram of a $p$-variate random field on $\mathbb{X} \times \mathbb{X}$. 
\end{theorem}

\section{Indicator Variogram of the Excursion Set of a Gaussian Random Field}
\label{excursion}
\vspace*{12pt}

Let $Z$ be a standard Gaussian random field on a set of points $\mathbb{X}$, with correlation function $\rho$. The following analytical expressions of the variogram $g_{\lambda}$ of $\lambda$-level excursion set of $Z$ are already known \citep{math75, Matheron1989}:
\begin{equation}
\label{matheron}
\begin{split}
    g_{\lambda} &= \frac{1}{2\pi} \int_{\rho}^1 \exp\left(-\frac{\lambda^2}{1+u}\right) \frac{{\rm d}u}{\sqrt{1-u^2}} \\
    &= \frac{1}{\pi} \exp\left(-\frac{\lambda^2}{2}\right) \sum_{k=1}^{+\infty} (1-\rho^k) \frac{H_{k-1}^2(\frac{\lambda}{\sqrt{2}})}{2^{k} k!}\\
    &= \frac{1}{2\pi \sqrt{2}} \exp\left(-\frac{\lambda^2}{2}\right) \sum_{k=0}^{+\infty} \frac{(-1)^k (1-\rho)^{k+\frac{1}{2}}}{k+\frac{1}{2}} \frac{H_{2k}(\frac{\lambda}{\sqrt{2}})}{8^k \, k!},
\end{split}
\end{equation}
where $H_{k}$ is the Hermite polynomial of degree $k$ \citep{szego1975}.

\medskip

Hereinafter, we present alternative analytical expressions of $g_{\lambda}$.

\subsection{First Alternative Expression}

Starting with the integral expression in (\ref{matheron}) and posing $t = \sqrt{\frac{1-u}{1+u}} = \tan v$, one obtains
\begin{equation*}
\begin{split}
    g_{\lambda} &= \frac{1}{\pi} \int_{0}^{\sqrt{\frac{1-\rho}{1+\rho}}} \exp\left(-\frac{\lambda^2(1+t^2)}{2}\right) \frac{{\rm d}t}{1+t^2} \\
    &= \frac{1}{\pi} \exp\left(-\frac{\lambda^2}{2}\right) \int_{0}^{\arctan\sqrt{\frac{1-\rho}{1+\rho}}} \exp\left(-\frac{\lambda^2 \tan^2 v}{2}\right) {\rm d}v.
\end{split}
\end{equation*}

\subsection{Second Alternative Expression}

Consider the random field $Y$ defined as
\begin{equation}
\label{nonergodicmodel}
    Y(x) = M + Z(x), \quad x \in \mathbb{X},
\end{equation}
where $M$ is a Gaussian random variable with mean $0$ and variance $\sigma^2$, independent of $Z$. As $Y$ is a zero-mean Gaussian random field with covariance function $\sigma^2 + \rho$, its median indicator variogram is (Proposition \ref{prop01})
\begin{equation}
\label{gamma0a}
    \gamma = \frac{1}{2\pi}\arccos \left(\frac{\sigma^2+\rho}{\sigma^2+1}\right).
\end{equation}
On the other hand, by conditioning the random field in (\ref{nonergodicmodel}) to the value of $M$, one can write this median indicator variogram as a mixture of indicator variograms of the excursion sets of $Z$, that is
\begin{equation}
\label{gamma0b}
\gamma = \frac{1}{\sigma\sqrt{2\pi}} \int_{-\infty}^{+\infty} \exp\left(-\frac{\lambda^2}{2\sigma^2} \right) g_\lambda {\rm d}\lambda.
\end{equation}
By equating (\ref{gamma0a}) and (\ref{gamma0b}), one finds
\begin{equation*}
     \frac{1}{\sqrt{\pi u}} \int_{0}^{+\infty} \exp\left(-\frac{\lambda^2}{4u} \right) g_\lambda {\rm d}\lambda = \frac{1}{2\pi}\arccos \left(\frac{2u+\rho}{2u+1}\right),
\end{equation*}
from which one derives the following expression of $g_\lambda$ \citep{Polyanin}, where $\mathcal{L}^{-1}$ stands for the inverse Laplace transform:
\begin{equation*}
    g_\lambda = \mathcal{L}^{-1} \left\{ \frac{p}{2\pi} \int_0^{+\infty} \exp(-p^2 t) \arccos \left(\frac{2t+\rho}{2t+1}\right) {\rm d}t \right\}[\lambda].
\end{equation*}

We invoke Post's inversion formula and the Leibniz integral rule for differentiation under the integral sign to write the inverse Laplace transform as
\begin{equation*}
    g_\lambda = \lim_{k\to +\infty} \frac{(-1)^{k}}{2\pi k!} \left(\frac{k}{\lambda}\right)^{k+1} \int_0^{+\infty} \frac{1}{-2t} f_t^{(k+1)}\left(\frac{k}{\lambda}\right) \arccos \left(\frac{2t+\rho}{2t+1}\right) {\rm d}t,
\end{equation*}
with $f_t: p \mapsto \exp(-p^2 t)$, that is: 
\begin{equation*}
\begin{split}
    g_\lambda &= \lim_{k\to +\infty} \frac{1}{2\pi k!} \int_0^{+\infty} \frac{1}{2t} \left(\frac{k \sqrt{t}}{\lambda}\right)^{k+1} \exp\left(\frac{k^2 t}{\lambda^2}\right) H_{k+1}\left(\frac{k \sqrt{t}}{\lambda}\right) \arccos \left(\frac{2t+\rho}{2t+1}\right) {\rm d}t\\
    &= \lim_{k\to +\infty} \frac{1}{2\pi k!}  \int_0^{+\infty} v^{k}  \exp(-v^2) H_{k+1}(v) \arccos \left(\frac{2v^2 \lambda^2/k^2+\rho}{2v^2\lambda^2/k^2+1}\right) {\rm d}v.
\end{split}
\end{equation*}

\section{Proofs}
\label{app_proof}
\vspace*{12pt}

The following lemma, which is of independent interest, will be required to prove Theorem \ref{thm1}.  

\begin{lemma}
\label{lem1}
    A necessary and sufficient condition for a real-valued symmetric matrix $\boldsymbol{D}$ of size $n \times n$ with zero diagonal entries to be the geodesic distance matrix of $n$ points of $\mathbb{S}^N(r)$ is that $r^2 \cos(\boldsymbol{D}/r)$ is the Gram matrix of these points, considered as vectors in $\mathbb{R}^{N+1}$, i.e., it is a positive semidefinite matrix with rank no greater than $N+1$ and with diagonal entries equal to $r^2$.  
\end{lemma}

\begin{proof}[Proof of Lemma \ref{lem1}]
    See Schoenberg \citep{Schoenberg1935}, Bogomolny et al. \citep{Bogomolny2008} or Wilson et al. \citep{Wilson2010}.
\end{proof}

\begin{proof}[Proof of Proposition \ref{closedconvex}]

\textit{Closedness.} One has to prove that, if $\{g_n: n \in \mathbb{N}\}$ is a sequence of indicator variograms tending pointwise to $g$ as $n$ tends to infinity, then $g$ is an indicator variogram. This can be done by noting that each $g_n$ satisfies the corner-positive inequalities (\ref{corner}), and so does the limit function $g$, which is therefore an indicator variogram.

\textit{Convexity.} One has to prove that, if $\{g_n: n \in \mathbb{N}\}$ is a sequence of indicator variograms and $\{\lambda_n: n \in \mathbb{N}\}$ is a sequence of nonnegative weights that add to $1$, then $g = \sum_{n \in \mathbb{N}} \lambda_n \, g_n$ is also an indicator variogram. To this end, let us denote by $Z_n$ an indicator random field with variogram $g_n$. Then, $g$ is the variogram of the indicator random field equal to $g_n$ with probability $\lambda_n$.

\end{proof}

\begin{proof}[Proof of Theorem \ref{thm1}]

\textit{Necessity.} Let $\mathbb{X}$ be a set of points, $(\Omega,\mathcal{A},P)$ a probability space, and $Z = \{Z(x,\omega): x \in \mathbb{X}, \omega \in \Omega \}$ a first-order stationary indicator random field on $\mathbb{X}$. 
Define the binary random field $\widetilde{Z} = 2Z-1$, which takes its values in $\{-1,1\}$. Such a random field has a finite expectation and a finite variance, therefore it possesses a noncentered covariance and a variogram. The latter two functions are defined as:
\begin{equation}
\label{Candg}
\begin{split}
    \widetilde{C}(x,y) &= \mathbb{E} [\widetilde{Z}(x,\cdot) \, \widetilde{Z}(y,\cdot)] = \int_{\Omega} \widetilde{Z}(x,\omega) \, \widetilde{Z}(y,\omega) \, P({\rm d}\omega), \quad x,y \in \mathbb{X}\\
    \widetilde{g}(x,y) &= \frac{1}{2} \mathbb{E} \left\{ [\widetilde{Z}(x,\cdot)-\widetilde{Z}(y,\cdot)]^2 \right\} = \frac{1}{2} \int_{\Omega} [\widetilde{Z}(x,\omega) - \widetilde{Z}(y,\omega)]^2 \, P({\rm d}\omega), \quad x,y \in \mathbb{X},
\end{split}
\end{equation}
and are related through
\begin{equation*}
    \label{Ctog}
    \widetilde{g}(x,y) = 1-\widetilde{C}(x,y)\quad x,y \in \mathbb{X}.
\end{equation*}

For $\omega \in \Omega$, define the mapping $\rho_{\omega}$ on $\mathbb{X} \times \mathbb{X}$ as
\begin{equation}
\label{covar}
    \rho_{\omega}(x,y) = \widetilde{Z}(x,\omega) \, \widetilde{Z}(y,\omega), \quad x,y \in \mathbb{X}.
\end{equation}
This is a symmetric and positive semidefinite mapping, insofar as 
$$\sum_{k=1}^n \sum_{\ell=1}^n \lambda_k \, \lambda_\ell \, \rho_{\omega}(x_k,x_\ell) = \left[\sum_{k=1}^n \lambda_k  \widetilde{Z}(x_k,\omega)\right]^2 \geq 0,$$
for any choice of the positive integer $n$, points $x_1, \ldots, x_n$ in $\mathbb{X}$ and real values $\lambda_1, \ldots, \lambda_n$. As $\rho_{\omega}(x,x)=1$ for any $x \in \mathbb{X}$, $\rho_{\omega}$ is a correlation function. Furthermore, since it takes its values in $\{-1,1\}$, one has $\rho_{\omega} = \frac{2}{\pi} \arcsin \rho_{\omega}$.

Accordingly, the noncentered covariance and the variogram in (\ref{Candg}) can be written as
\begin{equation}
\begin{split}
\label{Candg2}
    \widetilde{C}(x,y) &= \frac{2}{\pi} \int_{\Omega} \arcsin \rho_{\omega}(x,y) \, P({\rm d}\omega), \quad x,y \in \mathbb{X}\\
    \widetilde{g}(x,y) &= 1 - \widetilde{C}(x,y) = \frac{2}{\pi} \int_{\Omega} \arccos \rho_{\omega}(x,y) \, P({\rm d}\omega), \quad x,y \in \mathbb{X}.
\end{split}
\end{equation}
This part of the proof is concluded by noting that the variogram of $Z$ is $\frac{\widetilde{g}}{4}$ and that, under the assumption of separability, the sample space $\Omega$ and the probability function $P$ can be chosen as the interval  $(0,+\infty)$ and a cumulative distribution function on $(0,+\infty)$, respectively.

This result can be obtained in a straightforward manner by geometric arguments. Indeed, one can rewrite $\widetilde{g}$ in (\ref{Candg}) as
\begin{equation*}
\begin{split}
    \widetilde{g}(x,y) &= 4 \int_{\Omega} d_{\text{E}}({Z}(x,\omega),{Z}(y,\omega)) P({\rm d}\omega), \quad x,y \in \mathbb{X},
\end{split}
\end{equation*}
where $d_{\text{E}}$ is the Euclidean distance on $\{0,1\}$. It is known \citep{Critchley1997} that $\{0,1\}$ endowed with the Euclidean distance is isometrically embeddable into the circle $\mathbb{S}^1(\frac{1}{\pi})$ endowed with the geodesic distance $d_{\text{GC}}$. Using Lemma \ref{lem1}, one can therefore write
\begin{equation*}
    d_{\text{E}}({Z}(x,\omega),{Z}(y,\omega)) = \frac{1}{\pi} \arccos \rho_{\omega}(x,y), \quad x,y \in \mathbb{X}, \quad \omega \in \Omega,
\end{equation*}
for some correlation function $\rho_{\omega}$. Thus, $\widetilde{g}$ admits a representation as in (\ref{Candg2}).\\
\textit{Sufficiency.} For $t \in (0,+\infty)$, the Daniell-Kolmogorov extension theorem ensures the existence of a separable zero-mean Gaussian random field $Y_{t} = \{Y_{t}(x,\omega): x \in \mathbb{X}, \omega \in (0,+\infty)\}$ with $\rho_{t}$ as its covariance function. Let $T$ be a random variable on $(0,+\infty)$ with cumulative distribution function $F$, independent of the family of random fields $\{Y_{t}: t \in (0,+\infty)\}$. Define the indicator random field $Z = \{Z(x,\omega): x \in \mathbb{X}, \omega \in (0,+\infty)\}$ as
\begin{equation*}
    Z(x,\omega) = \mathsf{1}_{Y_T(x,\omega)>0} = \begin{cases}
        1 \text{ if $Y_T(x,\omega)>0$}\\
        0 \text{ otherwise.}
    \end{cases}
\end{equation*}
The mean value of $Z$ is $0.5$, hence $Z$ is first-order stationary, and its variogram is 
\begin{equation*}
\begin{split}
    g(x,y) &= \frac{1}{2} \mathbb{E}\{[Z(x,\cdot)-Z(y,\cdot)]^2\}\\
    &= \frac{1}{2}  \mathbb{E}\{[\mathsf{1}_{Y_T(x,\cdot)>0} - \mathsf{1}_{Y_T(y,\cdot)>0}]^2\}\\
    &= \frac{1}{2} \mathbb{E}\left\{\mathbb{E}\{ [\mathsf{1}_{Y_T(x,\cdot)>0}- \mathsf{1}_{Y_T(y,\cdot)>0}]^2 \mid T\} \right\}.
\end{split}
\end{equation*}
Owing to Proposition \ref{prop01}, this is
\begin{equation*}
\begin{split}
    g(x,y) &= \frac{1}{2\pi} \mathbb{E}\{\arccos \rho_T(x,y)\}\\
    &= \frac{1}{2\pi} \int_0^{+\infty} \arccos \rho_t(x,y) F({\rm d}t).
\end{split}
\end{equation*}
Therefore, the variogram of $Z$ is of the form (\ref{indicvariog}), which concludes the proof. 
\end{proof}

\begin{proof}[Proof of Corollary \ref{hull}]
    The proof stems from Theorem \ref{thm1} and Propositions \ref{prop01} and \ref{closedconvex}.
\end{proof}

\begin{proof}[Proof of Corollary \ref{cor01}]
The proof stems from the fact that the random field $Y_T$ introduced in the proof of Theorem \ref{thm1} is a Hermitian random field.    
\end{proof}

\begin{proof}[Proof of Remark \ref{rem2}]
    Owing to the Daniell-Kolmogorov extension theorem, for any correlation function $\rho_{\omega}$ on $\mathbb{X} \times \mathbb{X}$, the mapping $\gamma_{\omega}=1-\rho_{\omega}$ is the variogram of a standard Gaussian random field on $\mathbb{X}$. The proof then relies on the following relation between arccosine and arcsine functions that can be derived from formulae 1.623.1, 1.625.6 and 1.626.2 of Gradshteyn and Ryzhik \citep{Grad}:
    $$\arccos (1-2t^2) = 2 \arcsin t,\quad t \in [0,1].$$
\end{proof}

\begin{proof}[Proof of Proposition \ref{thm3}]
    Let $Z$ and $Z^\prime$ be two mutually independent indicator random fields on $\mathbb{X}$ with respective variograms $g$ and $g^\prime$, and let $\varpi \in [0,1]$. Then:
    \begin{enumerate}
        \item $\varpi \,g$ is the variogram of the indicator random field defined by
        \begin{equation*}
         \widetilde{Z} = \begin{cases}
            Z \text{ with probability $\varpi$}\\
            0 \text{ with probability $1-\varpi$.}            
        \end{cases}           
        \end{equation*}
        \item $\varpi \, g + (1-\varpi) \, g^\prime$ is the variogram of the indicator random field defined by
        \begin{equation*}
         \widetilde{Z} = \begin{cases}
            Z \text{ with probability $\varpi$}\\
            Z^\prime \text{ with probability $1-\varpi$.}            
        \end{cases}           
        \end{equation*}
        \item $g+g'-4 g \, g^\prime$ is the variogram of the indicator random field $\frac{1+\widetilde{Z} \widetilde{Z^\prime}}{2}$, where
    \begin{equation*}
        \widetilde{Z} = \begin{cases}
            2Z - 1 \text{ with probability $\frac{1}{2}$}\\
            1-2Z \text{ with probability $\frac{1}{2}$}            
        \end{cases}
    \end{equation*}
    and
    \begin{equation*}
        \widetilde{Z^\prime} = \begin{cases}
            2Z^\prime - 1 \text{ with probability $\frac{1}{2}$}\\
            1-2Z^\prime \text{ with probability $\frac{1}{2}$.}            
        \end{cases}
    \end{equation*}
    \end{enumerate}
    
\end{proof}

\begin{proof}[Proof of Theorem \ref{g2exp}]
    Owing to Proposition \ref{thm3}, it suffices to prove the claim of the theorem for $\varpi=1$. Let $Z$ be an indicator random field with variogram $g$. The random field $\widetilde{Z}$ defined by
    \begin{equation*}
        \widetilde{Z} = \begin{cases}
            2Z - 1 \text{ with probability $\frac{1}{2}$}\\
            1-2Z \text{ with probability $\frac{1}{2}$}            
        \end{cases}
    \end{equation*}
    takes the values $-1$ or $1$ with equal probability, i.e., it has a Rademacher marginal distribution. Its expected value is $0$, its variance is $1$, its variogram is $4g$, and its covariance function is $\widetilde{C} = 1-4g$. For $t>0$, one has
    \begin{equation*}
        \exp(-4tg)=\sum_{k=0}^{+\infty} \frac{\exp(-t) t^k}{k!} \widetilde{C}^k,
    \end{equation*}
    where $\widetilde{C}^0$ is the covariance function of the random field $\widetilde{Z}_0$ equal to $1$ or to $-1$ everywhere with equal probabilities $\frac{1}{2}$ and, for $k>0$, $\widetilde{C}^k$ is the covariance function of the random field $\widetilde{Z}_k$ with Rademacher marginal distribution obtained by multiplying $k$ independent copies of $\widetilde{Z}$. Accordingly, $\exp(-4tg)$ is the covariance function of the random field $\widetilde{Z}^\prime$ equal to $\widetilde{Z}_k$ with probability $\exp(-t) t^k/k!$, and $1-\exp(-4tg)$ is its variogram. It is concluded that $\frac{1}{4}(1-\exp(-4tg))$ is the variogram of the indicator random field $(1+\widetilde{Z}^\prime)/2$.
    
    \end{proof}

\begin{proof}[Proof of Theorem \ref{thm4}]
    The corner-positive inequalities (\ref{corner}) are necessary and sufficient for $g$ to be the variogram of a first-order stationary indicator random field \citep{Quintanilla2008}, while the negative-type and upper-bound inequalities (\ref{condneg}) and (\ref{upperbound}) are necessary. 
    All the remaining inequalities (\ref{polyg}) to (\ref{roundedpsd}) are implied by the gap inequalities \citep{Galli2012}. To prove the latter inequalities for an indicator variogram, consider an indicator random field $Z$ on $\mathbb{X}$ with variogram $g$, and define 
    \begin{equation*}
        \widetilde{Z} = \begin{cases}
            2Z - 1 \text{ with probability $\frac{1}{2}$}\\
            1-2Z \text{ with probability $\frac{1}{2}$.}            
        \end{cases}
    \end{equation*}
    The random field $\widetilde{Z}$ takes its values in $\{-1,1\}$, has a zero expected value and its covariance function is $\widetilde{C} = 1-4g$. For any $\omega \in \Omega$ and any set of weights $\lambda_1, \ldots, \lambda_n \in \mathbb{Z}$, one has
    \begin{equation*}
        \sum_{k = 1}^n \sum_{\ell = 1}^n \lambda_k \, \lambda_\ell \, \widetilde{Z}(x_k,\omega) \, \widetilde{Z}(x_\ell,\omega) = \left[\sum_{k=1}^n \lambda_k \, \widetilde{Z}(x_k,\omega) \right]^2 \geq \gamma(\boldsymbol{\lambda})^2.
    \end{equation*}
    Taking the expected values, it comes
    \begin{equation*}
        \gamma(\boldsymbol{\lambda})^2 \leq \sum_{k = 1}^n \sum_{\ell = 1}^n \lambda_k \, \lambda_\ell \, \widetilde{C}(x_k,x_\ell) = \sigma(\boldsymbol{\lambda})^2 - 4 \sum_{k = 1}^n \sum_{\ell = 1}^n \lambda_k \, \lambda_\ell \, g(x_k,x_\ell).
    \end{equation*}
    
\end{proof}

\begin{proof}[Proof of Theorem \ref{thm2}]
\textit{Necessity.} For this part of the proof, we use geometric arguments as in the proof of Theorem \ref{thm1}. Let $\mathbb{X}$ be a set of points and $Z = \{Z(x,\omega): x \in \mathbb{X}, \omega \in \Omega\}$ be a random field on $\mathbb{X}$ having a madogram $\gamma^{(1)}$. The latter is defined as:
\begin{equation*}
\begin{split}
    \gamma^{(1)}(x,y) &= \frac{1}{2} \mathbb{E} \left\{ \lvert Z(x,\cdot)-Z(y,\cdot) \rvert \right\} \\
    &=  \frac{1}{2} \int_{\Omega}  d_{\text{E}}(Z(x,\omega),Z(y,\omega)) P({\rm d}\omega), \quad x,y \in \mathbb{X},
\end{split}
\end{equation*}
where $d_{\text{E}}$ is the Euclidean distance on $\mathbb{R}$. 

For any fixed $\omega \in \Omega$, the squared Euclidean distance $(x,y) \mapsto d^2_{\text{E}}(Z(x,\omega),Z(y,\omega))$ is a conditionally negative semidefinite mapping on $\mathbb{X} \times \mathbb{X}$ \citep{Schoenberg1938, Bogomolny2008} that is symmetric and vanishes on the diagonal of $\mathbb{X} \times \mathbb{X}$. Therefore, one can write
\begin{equation*}
    d_{\text{E}}(Z(x,\omega),Z(y,\omega)) = 2 \sqrt{\gamma_{\omega}(x,y)}, \quad x,y \in \mathbb{X}, \quad \omega \in \Omega,
\end{equation*}
for some variogram $\gamma_{\omega}$ on $\mathbb{X} \times \mathbb{X}$. 

Under the assumption of separability, one can choose $\Omega=(0,+\infty)$ and $P=F$ and deduces that $\gamma^{(1)}$ can be expanded as in (\ref{madogram}).\\
\textit{Sufficiency.} 
Owing to the Daniell-Kolmogorov extension theorem, for any $t \in (0,+\infty)$, the mapping $\sqrt{\gamma_{t}}$ is the madogram of a separable Gaussian random field $Y_{t} = \{Y_t(x,\omega): x \in \mathbb{X}, \omega \in (0,+\infty)\}$ with variogram ${\pi \gamma_{t}}$ (Proposition \ref{prop02}). Accordingly, the mapping $\gamma^{(1)}$ in (\ref{madogram}) is the madogram of the random field $Z = \{Z(x,\omega): x \in \mathbb{X}, \omega \in (0,+\infty)\}$ defined as
$$Z(x,\omega) = Y_T(x,\omega), \quad x \in \mathbb{X}, \quad \omega \in (0,+\infty),$$
with $T$ a random variable on $(0,+\infty)$ with cumulative distribution function $F$, independent of the family of random fields $\{Y_t: t \in (0,+\infty)\}$.
\end{proof}

\begin{proof}[Proof of Corollary \ref{cor2}]
    The claim stems from Theorem \ref{thm2} and the fact that the variogram of a first-order stationary indicator random field is also its madogram. 
\end{proof}

\begin{proof}[Proof of Theorem \ref{mad2ind}]
The variogram of a first-order stationary indicator random field is also its madogram, which concludes the ``if'' part of the proof. Reciprocally, let $\gamma^{(1)}: \mathbb{X} \times \mathbb{X} \to [0,\frac{1}{2}]$ be the madogram of a random field $Z=\{Z(x,\omega): x \in \mathbb{X}, \omega \in \Omega\}$ with values in $[0,1]$ and stationary marginal distribution. The excursion set indicator associated with a threshold $\lambda$ is therefore first-order stationary, and its variogram is
\begin{equation*}
\begin{split}
    g_{\lambda}(x,y) &= \frac{1}{2} \mathbb{E}\{[ \mathsf{1}_{Z(x,\cdot)\geq \lambda} - \mathsf{1}_{Z(y,\cdot)\geq \lambda}]^2 \}\\
    &= \frac{1}{2} \mathbb{E}\{| \mathsf{1}_{Z(x,\cdot)\geq \lambda} - \mathsf{1}_{Z(y,\cdot)\geq \lambda}| \}\\
    &= \frac{1}{2} \int_{\Omega} | \mathsf{1}_{Z(x,\cdot)\geq \lambda} - \mathsf{1}_{Z(y,\cdot)\geq \lambda}| P({\rm d}\omega), \quad x, y \in \mathbb{X}.
\end{split}
\end{equation*}
By integrating over all the thresholds and applying Fubini's theorem, one finds \citep{Matheron1989}:
\begin{equation*}
\begin{split}
    \int_0^1 g_{\lambda}(x,y) {\rm d}\lambda &= \frac{1}{2} \int_{\Omega} \int_0^1 | \mathsf{1}_{Z(x,\omega)\geq \lambda} - \mathsf{1}_{Z(y,\omega)\geq \lambda}| {\rm d}\lambda \, P({\rm d}\omega) \\
    &= \frac{1}{2} \int_{\Omega} \int_{\min \{Z(x,\omega),Z(y,\omega)\}}^{\max \{Z(x,\omega),Z(y,\omega)\}} {\rm d}\lambda \, P({\rm d}\omega) \\
    &= \frac{1}{2} \int_{\Omega} |Z(x,\omega)-Z(y,\omega)| \, P({\rm d}\omega) \\
    &= \gamma^{(1)}(x,y), \quad x, y \in \mathbb{X}.
\end{split}
\end{equation*}

Owing to (\ref{corner}), for any $n > 1$ and any corner-positive matrix $[\lambda_{k\ell}]_{k,\ell=1}^n$, one has
\begin{equation*}
    \int_0^1 \sum_{k=1}^n \sum_{\ell=1}^n \lambda_{k\ell} [1 - 4 g_{\lambda}(x_k,x_\ell)] \, {\rm d}\lambda \geq 0,
\end{equation*}
that is
\begin{equation*}
    \sum_{k=1}^n \sum_{\ell=1}^n \lambda_{k\ell} [1 - 4 \gamma^{(1)}(x_k,x_\ell)] \geq 0.
\end{equation*}
The madogram $\gamma^{(1)}$ is a conditionally negative semidefinite mapping on $\mathbb{X} \times \mathbb{X}$ that vanishes on the diagonal and satisfies the corner-positive inequalities (\ref{corner}). It is therefore the variogram of a first-order stationary indicator random field.
\end{proof}

\begin{proof}[Proof of Corollary \ref{mad2ind2}]
    This is an immediate consequence of Theorem \ref{mad2ind}.
\end{proof}

\begin{proof}[Proof of Theorem \ref{thm5}]
    The inequalities in the first part of the theorem are derived in a similar fashion as in Theorem \ref{thm4}, based on the fact that any madogram is, up to a positive constant, the limit of a sequence of indicator variograms \citep{Matheron1989} and that the polygonal, negative-type, odd-clique, and hypermetric inequalities in Theorem \ref{thm4} are insensitive to the multiplication of function $g$ by a positive constant.

    The second part of the theorem, where the random field is assumed to be bounded and to have a stationary marginal distribution, derives from Theorems \ref{thm4} and \ref{mad2ind}.

\end{proof}

\begin{proof}[Proof of Example \ref{example3}]
     For $a>0$, define the triangular wave ${\cal T}_a$ as
    \begin{equation}
    \label{triangwave}
        {\cal T}_a(t) = \frac{1}{2\pi} \arccos \cos(a\, t), \quad t \in \mathbb{R}.
    \end{equation}
    Then, ${\cal T}_a(d_{\text{E}})$ is a valid indicator variogram on $\mathbb{R} \times \mathbb{R}$ owing to Proposition \ref{prop01} and to the fact that the cosine wave $(x,y) \mapsto \cos(a d_{\text{E}}(x,y))$ is a valid covariance on $\mathbb{R} \times \mathbb{R}$. 
    Let now $Z_a$ be an indicator random field in $\mathbb{R}$ with mean value $0.5$ and covariance function $\rho_a = \frac{1}{4} - {\cal T}_a(d_{\text{E}})$. Define an indicator random field $\widetilde{Z}_a$ in $\mathbb{R}^N$ with mean $0.5$ as
    \begin{equation*}
        \widetilde{Z}_a(x) = Z_a(\langle x,U \rangle), \quad x \in \mathbb{R}^N,
    \end{equation*}
    with $\langle \cdot, \cdot \rangle$ the inner product, and $U$ a uniform point on the sphere $\mathbb{S}^{N-1}(1)$. The covariance function $\widetilde{\rho}_a$ of $\widetilde{Z}_a$ is related to $\rho_a$ by the turning bands operator \citep{Matheron1973}:
    \begin{equation*}
        \widetilde{\rho}_a(x,y) = \mathbb{E}\left[\frac{1}{4} - {\cal T}_a(\langle x-y, U \rangle) \right], \quad x,y \in \mathbb{R}^N.
    \end{equation*}
    We expand the triangular wave into a Fourier series \citep{Oldham2009} to obtain
    \begin{equation*}
        \frac{1}{4} - {\cal T}_a(t) = \frac{2}{\pi^2} \sum_{k=0}^{+\infty} \frac{\cos((2k+1)a t)}{(2k+1)^2}, \quad t \geq 0,
    \end{equation*}
    and, therefore,
    \begin{equation*}
        \widetilde{\rho}_a(x,y) = \frac{2}{\pi^2} \sum_{k=0}^{+\infty} \frac{\mathbb{E}[\cos((2k+1)a \langle x-y, U \rangle)]}{(2k+1)^2}, \quad x,y \in \mathbb{R}^N.
    \end{equation*}
    The turning bands operator applied to the cosine covariance provides the J-Bessel covariance (denoted $\mathbb{J}_N$) \citep{Matheron1973}, i.e.
    \begin{equation*}
        \widetilde{\rho}_a(x,y) = \frac{2}{\pi^2} \sum_{k=0}^{+\infty} \frac{\mathbb{J}_N((2k+1) \, a d_{\text{E}}(x,y))}{(2k+1)^2}, \quad x,y \in \mathbb{R}^N.
    \end{equation*}
    Since the radial part $\rho$ of an isotropic correlation function in $\mathbb{R}^N \times \mathbb{R}^N$ is a scale mixture of such J-Bessel covariances, it is seen that, by randomizing the scale parameter $a$, one can construct an indicator random field in $\mathbb{R}^N$ with mean $0.5$ and covariance $$\frac{2}{\pi^2} \sum_{k=0}^{+\infty} \frac{\rho((2k+1) \, d_{\text{E}}(x,y))}{(2k+1)^2} = \frac{1}{4} - \frac{2}{\pi^2} \sum_{k=0}^{+\infty} \frac{\gamma((2k+1) d_{\text{E}}(x,y))}{(2k+1)^2}, \quad x,y \in \mathbb{R}^N,$$
    where $\gamma = 1-\rho$ is the radial part of the variogram of an isotropic standard Gaussian random field. The result follows from Proposition \ref{thm3} and the fact that the variogram of an indicator random field with mean $0.5$ is $\frac{1}{4}$ minus the covariance.
\end{proof}

\begin{proof}[Proof of Table \ref{tab:euclideanexamples}]
    On account of formulae 5.1.26.16 and 5.1.26.28 of Prudnikov et al. \citep{prud1}, the hyperbolic tangent models turn out to be particular cases of (\ref{seriesmodel}), with $\gamma$ being a Cauchy variogram:
    \begin{equation}
    \label{cauchy}
        \gamma(x,y)=1-\left(1+\frac{\pi^2 d_{\text{E}}(x,y)^2}{4\lambda^2}\right)^{-\beta}, \quad x,y \in \mathbb{R}^N,
    \end{equation}
    and $\beta=1$ (model 1) or $\beta=2$ (model 2). 

    Concerning the I-Bessel model, for $a>0$, we define the quadratic wave ${\cal Q}_a$ as the periodic repetition of the following mapping on $[0,a]$:
    \begin{equation*}
        \widetilde{{\cal Q}}_a(t) = \frac{3t}{2a} \left(1-\frac{t}{a}\right), \quad t \in [0,a].
    \end{equation*}
    Then, ${\cal Q}_a$ is a scale mixture of triangular waves of the form (\ref{triangwave}) with weights adding to $1$ \citep{Tabor2009}, so that ${\cal Q}_a(d_{\text{E}})$ is a valid indicator variogram on $\mathbb{R} \times \mathbb{R}$. The Laplace transform of ${\cal Q}_a$ is \citep{prud4}
    \begin{equation*}
        \int_0^{+\infty} {\cal Q}_a(t) \, \exp(-st) {\rm d}t = \frac{3 \sqrt{\pi}}{2a^2 (1-\exp(-as))} \left(\frac{a}{s}\right)^{\frac{3}{2}} \exp\left(-\frac{as}{2}\right) I_{\frac{3}{2}}\left(\frac{as}{2}\right), \quad s \in [0,+\infty),
    \end{equation*}
    which leads to the following identity:    
    \begin{equation*}
        \lambda \int_0^{+\infty} {\cal Q}_{1/u}(t) \, \exp(-\lambda u) {\rm d}u = \frac{3 \sqrt{\frac{\pi t}{\lambda}}}{2-2\exp(-\frac{\lambda}{t})} \exp\left(-\frac{\lambda}{2t}\right) I_{\frac{3}{2}}\left(\frac{\lambda}{2t}\right), \quad t \in [0,+\infty).
    \end{equation*}
    Accordingly, up to the factor $\varpi$, the I-Bessel model is an exponential scale mixture of valid indicator variograms on $\mathbb{R} \times \mathbb{R}$ and is itself an indicator variogram on $\mathbb{R} \times \mathbb{R}$, as per Theorem \ref{thm1} and Proposition \ref{thm3}.  
\end{proof}

\begin{proof}[Proof of Example \ref{madex}]
Consider a random field of the form
\begin{equation*}
    Z(x) = \frac{1+\cos(2\langle U,x\rangle+\phi)}{2}, \quad x \in \mathbb{R}^N,
\end{equation*}
where $U$ is a random vector in $\mathbb{R}^N$ and $\phi$ is an independent random variable uniformly distributed in $[0,2\pi]$. The madogram of $Z$ is
\begin{equation*}
\begin{split}
    \gamma^{(1)}(x,y) &= \frac{1}{4} \mathbb{E}\{|\cos(2\langle U,x\rangle+\phi) - \cos(2\langle U,y\rangle+\phi)|\}\\
    &= \frac{1}{2} \mathbb{E}\{|\sin(\langle U,x+y\rangle+\phi) \sin(\langle U,x-y\rangle)|\}\\
    &= \frac{1}{\pi} \mathbb{E}\{|\sin(\langle U,x-y\rangle)|\}, \quad x, y \in \mathbb{R}^N.
\end{split}
\end{equation*}
Using the Fourier expansion of the absolute sine, it comes:
\begin{equation*}
\begin{split}
    \gamma^{(1)}(x,y) &= \frac{1}{\pi} \sum_{k=-\infty}^{+\infty} \frac{2}{\pi(1-4k^2)} \mathbb{E}\{\exp(2\mathsf{i}k \langle U,x-y\rangle)\}, \quad x, y \in \mathbb{R}^N,
\end{split}
\end{equation*}
where $\mathsf{i}$ is the imaginary unit. If $U$ has the spectral distribution of a stationary correlation function $\rho$, then
\begin{equation*}
\begin{split}
    \gamma^{(1)}(x,y) &= \frac{2}{\pi^2} + \frac{4}{\pi^2} \sum_{k=1}^{+\infty} \frac{\rho(2k(x-y))}{1-4k^2}, \quad x, y \in \mathbb{R}^N,
\end{split}
\end{equation*}
which is equal to zero when $x=y$ owing to formula 5.1.25.4 of Prudnikov et al. \citep{prud1}. The claim of the example stems from Proposition \ref{thm3}, Theorem \ref{mad2ind} and the identity $\rho=1-\gamma$.
\end{proof}

\begin{proof}[Proof of Example \ref{complmonot}]
    Owing to (\ref{berns}), one has 
    \begin{equation}
    \label{bernstein}
        \varphi(x) + 1-\varphi(0) = \int_0^{+\infty} \exp(-t x) {\rm d}\widetilde{F}(t),
    \end{equation}    
    with ${\rm d}\widetilde{F}(t) = (1-\varphi(0)) \, \delta_0(t) + \varphi(0) {\rm d}F(t)$, where $F$ is the cumulative distribution function of a probability measure on $[0,+\infty)$ and $\delta_0$ is the Dirac measure at $t=0$. Accordingly, $\widetilde{F}$ is the cumulative distribution function of a probability measure on $[0,+\infty)$. The claim of the Example stems from Theorem \ref{thm1} and the fact that, for any $t>0$, the mapping $\rho_t: (x,y) \mapsto \sin(\frac{\pi}{2}\exp(-t d_{\text{E}}(x,y)))$ is a correlation function on $\mathbb{R}^N \times \mathbb{R}^N$ \citep{Lantuejoul2002}. 
\end{proof}

\begin{proof}[Proof of Example \ref{linexample}]
Owing to Lemma \ref{lem1}, the mapping $\rho$ defined by
\begin{equation*}
    \rho_t(x,y) = \cos\left(\frac{d_{\text{GC}}(x,y)}{r}\right), \quad x, y \in \mathbb{S}^N(1),
\end{equation*}
is a covariance function. The result follows from Propositions \ref{prop01} and \ref{thm3}. 
\end{proof}

\begin{proof}[Proof of Example \ref{expexample}]
The proof stems from Example \ref{linexample} and Theorem \ref{g2exp}. 
\end{proof}

\begin{proof}[Proof of Example \ref{triexample}]
    The mapping $g$ can be rewritten as $$g = \frac{\varpi}{2\pi} \arccos \cos(k\, d_{\text{GC}}) = \frac{\varpi}{2\pi} \arccos T_k(\cos d_{\text{GC}}),$$ where $T_k$ is the Chebyshev polynomial of the first kind of degree $k$. The claim stems from Propositions \ref{prop01} and \ref{thm3} and the fact that $T_k(\cos d_{\text{GC}})$ is a valid covariance on the unit circle \citep{Schoenberg1942}. 
\end{proof}

\begin{proof}[Proof of Example \ref{quadexample}]
    The mapping $g$ can be written as a mixture of triangular waves of the form (\ref{triangwav}) with weights that add to $1$ \citep{Tabor2009}, therefore it is a valid indicator variogram owing to Theorem \ref{thm1} and Proposition \ref{thm3}.
\end{proof}

\begin{proof}[Proof of Example \ref{examplegraph1}]
This stems from the fact that the referred models are valid indicator variograms in Euclidean spaces of any dimension, and that the square-rooted resistance distance \citep{Klein1998} and the communicability distance \citep{Estrada2012} are Euclidean distances.
\end{proof}

\begin{proof}[Proof of Example \ref{nugget}]
Owing to Proposition \ref{thm3}, it suffices to establish the result for $\varpi=1$. The proof relies on a straightforward generalization to an abstract set of points $\mathbb{X}$ of the construction proposed by McMillan  \citep{McMillan1955} for random fields on $\mathbb{Z}$. The case when $X$ is the real line has been examined by Shamai and Chaya \citep{Shamai1991}.
\end{proof}

\begin{proof}[Proof of Example \ref{generalexamples}] 
\textbf{Model 1.} 
Let $Z_1,\ldots,Z_k$ be $k$ independent copies of a Gaussian random field on $\mathbb{X}$ with variogram $a \gamma$. Let $I_1,\ldots,I_k$ be $k$ independent copies of an indicator random field on $\mathbb{R}$ with exponential covariance $\frac{1}{4}\exp(-d_\text{E})$ (Example \ref{complmonot}). Define the indicator random field
    \begin{equation*}
        \widetilde{I}(x,\omega) = \frac{1}{2} \left[1+\prod_{i=1}^k \left(2I_i(Z_i(x,\omega))-1\right)\right], \quad x \in \mathbb{X}, \omega \in \Omega.
    \end{equation*}
    Each factor in the above product is a binary ($\pm 1$) random field with a covariance given by \citep{Grad}
    \begin{equation*}
    \begin{split}
        \mathbb{E}\left[ \exp(-\lvert Z_i(x,\omega) - Z_i(y,\omega) \rvert) \right] &= \sqrt{\frac{2}{\pi}} \int_{0}^{\infty} \exp(-\sqrt{2a\gamma(x,y)} u) \exp\left(-\frac{u^2}{2}\right) {\rm d}u\\
        &= \exp(a\gamma(x,y)) \left[ 1-\text{erf}(\sqrt{a\gamma(x,y)})\right], \quad x,y\in \mathbb{X}.
    \end{split}
    \end{equation*}
    This covariance raised to power $k$ gives the covariance of $\widetilde{I}$, which is four times the covariance of the indicator random field $\frac{1+\widetilde{I}}{2}$. We invoke Proposition \ref{thm3} to conclude the proof of the first entry of the table.  
    
\noindent \textbf{Model 2.}
    The proof follows the same line of reasoning as in the previous model, with a tent covariance instead of an exponential one for the indicator random fields $I_1, \ldots, I_k$. The tent covariance is a valid indicator covariance in $\mathbb{R} \times \mathbb{R}$ \citep{Matheron1988}.

\noindent \textbf{Models 3 to 6.} 
    Let $F$ be the cumulative distribution function of a random variable on $[0,1]$, $k$ be a positive integer, and $\rho$ be a correlation function on $\mathbb{X} \times \mathbb{X}$. For $t \in [0,1]$, define the correlation function
    \begin{equation}
    \label{tk}
        \rho_t = 1 - t^k \gamma,
    \end{equation}
    with $\gamma = 1-\rho$ being the variogram of a standard Gaussian random field on $\mathbb{X}$. Based on Theorem \ref{thm1}, the following mapping is an indicator variogram:
    \begin{equation*}
        \begin{split}
            g(x,y) = \frac{1}{2\pi} \int_0^1 \arccos (1 - t^k \gamma(x,y)) \, {\rm d}F(t), \quad x,y\in \mathbb{X}.
        \end{split}
    \end{equation*}
    An integration by part gives
    \begin{equation}
    \label{IBP1}
        \begin{split}
            g(x,y) = \frac{1}{2\pi} \sqrt{\gamma(x,y)} \int_0^1 \frac{t^{\frac{k}{2}-1} [1-F(t)]}{\sqrt{2-t^k \, \gamma(x,y)}} {\rm d}t, \quad x,y\in \mathbb{X}.
        \end{split}
    \end{equation}

    \begin{enumerate}
    \item Let $k=1$ and consider the cumulative distribution function
    \begin{equation*}
        F(t) = 1 - (1-t)^{\lambda-1} \, (1-\beta t)^{-\alpha}, \quad t \in [0,1],
    \end{equation*}
    with $\lambda>0$, $\beta \in (-1,1]$ and $(\alpha,\beta,\lambda)$ such that $F$ is non-decreasing. Then, based on formula 2.2.8.5 of Prudnikov et al. \citep{prud1}, one obtains
    \begin{equation*}
        \begin{split}
            g(x,y) = \sqrt{\frac{\gamma(x,y)}{2\pi}} \frac{\Gamma(\lambda)}{2 \Gamma(\frac{1}{2}+\lambda)} F_1\left(\frac{1}{2};\frac{1}{2},\alpha;\frac{1}{2}+\lambda;{\frac{\gamma(x,y)}{2}},\beta\right), \quad x,y\in \mathbb{X},
        \end{split}
    \end{equation*}
    where $F_1$ is an Appell hypergeometric function. Model 3 is obtained by considering the cases $\beta=0$ and $\lambda \geq 1$, $\beta \in (0,1)$ and $\lambda = \alpha \geq \frac{1}{1-\beta}$, $\beta \in (-1,0)$ and $\lambda = \alpha \geq 1$, and $\beta=1$ and $\lambda-\alpha \geq 1$, for which reduction formulae of the Appell function exist \citep{Erdelyi1}.
    \item Let $k=1$ and consider the cumulative distribution function
    \begin{equation*}
        F(t) = 1 - (1-t)^{\lambda-1} \, {}_2F_1\left(\alpha,\beta;\frac{1}{2};t\right), \quad t \in [0,1],
    \end{equation*}
    with $\alpha$, $\beta>0$ and $\lambda>0$ such that $F$ is well defined and non-decreasing. A sufficient condition is that $\alpha \in (\frac{1}{2},\frac{3}{2})$, $\beta \in (\frac{1}{2}-\alpha,\frac{1}{2})$ and $\lambda \geq \alpha+\beta+\frac{1}{2}$, in which case the mapping $t \mapsto (1-t)^{\alpha+\beta-\frac{1}{2}} \, {}_2F_1\left(\alpha,\beta;\frac{1}{2};t\right)$ is non-increasing \citep{Olver} and so is $1-F$. Model 4 follows from formula 7.512.9 of Gradshteyn and Ryzhik \citep{Grad}. 
    \item Consider the cumulative distribution function
    \begin{equation*}
        F(t) = 1 - (1-t)^{\lambda-1}, \quad t \in [0,1],
    \end{equation*}
    with $\lambda > 1$. Then, based on formula 2.1.1.3 of Exton \citep{Exton1978}, one obtains
    \begin{equation*}
    \begin{split}
        g(x,y) &= \frac{1}{\pi} \sqrt{\frac{\gamma(x,y)}{2}} \, \frac{\Gamma(\frac{k}{2}) \Gamma(\lambda)}{\Gamma(\frac{k}{2}+\lambda)} \\&\times {}_{k+1}F_k\left(\frac{1}{2},\frac{1}{2},\ldots,\frac{1}{2}+\frac{k-1}{k};\frac{1}{2}+\frac{\lambda}{k},\ldots,\frac{1}{2}+\frac{\lambda+k-1}{k};{\frac{\gamma(x,y)}{2}}\right), \quad x,y\in \mathbb{X}.
    \end{split}
    \end{equation*}
    Model 5 corresponds to the case when $k=2$. 

    \item Let $k=1$. Expanding the integrand of (\ref{IBP1}) by means of Newton's generalized binomial theorem and integrating by part leads to:
    \begin{equation*}
        \begin{split}
            g(x,y) = \frac{1}{2\pi} \sqrt{\frac{\gamma(x,y)}{2\pi}} \sum_{n=0}^{+\infty} \frac{\Gamma(n+\frac{1}{2})}{(n+\frac{1}{2}) n!} \left(\frac{\gamma(x,y)}{2}\right)^n \int_0^1 t^{n+\frac{1}{2}} {\rm d}F(t), \quad x,y\in \mathbb{X}.
        \end{split}
    \end{equation*}
    Let $F$ be the cumulative distribution function of the beta distribution with parameters $(\alpha,\beta)$. Then:
    \begin{equation*}
        \begin{split}
            g(x,y) &= \frac{1}{2\pi} \sqrt{\frac{\gamma(x,y)}{2\pi}} \sum_{n=0}^{+\infty} \frac{\Gamma(n+\frac{1}{2})}{(n+\frac{1}{2}) n!} \left(\frac{\gamma(x,y)}{2}\right)^n \frac{\Gamma(\alpha+n+\frac{1}{2}) \Gamma(\alpha+\beta)}{\Gamma(\alpha) \Gamma(\alpha+\beta+n+\frac{1}{2})}, \quad x,y\in \mathbb{X},
        \end{split}
    \end{equation*}
    which yields Model 6.
    \end{enumerate}
 \end{proof}

\begin{proof}[Proof of Theorem \ref{orderineq}]
    For $\alpha \leq 1$, any variogram of order $\alpha$ is also a madogram, thus the inequalities of Theorem \ref{thm5} apply. For $\alpha \geq 1$, the inequality derives from the definition of the variogram of order $\alpha$ and from the Minkowski inequality. 
\end{proof}

\begin{proof}[Proof of Theorem \ref{order2ind}]
    Any first-order stationary indicator variogram is also a variogram of order $\alpha$. Reciprocally, any variogram of order $\alpha \leq 1$ is also a madogram, so that, under the conditions of the theorem, it is a first-order stationary indicator variogram owing to Theorem \ref{mad2ind}.
\end{proof}

\begin{proof}[Proof of Corollary \ref{alpha2beta}]
Let $\alpha, \beta \in (0,1]$ and let $Z$ be a random field with bounded values and stationary marginal distribution. Without loss of generality, assume that the support of this distribution is $[-\frac{a}{2},\frac{a}{2}]$, with $a>0$. Then, the variogram of order $\alpha$ of $Z$ is $a^\alpha$ times the variogram of order $\alpha$ of $\frac{Z}{a}+\frac{1}{2}$. Owing to Theorem \ref{order2ind}, the latter is an indicator variogram and, therefore, the variogram of order $\beta$ of a random field $Y$ with marginal distribution supported in $[0,1]$. Accordingly, the variogram of order $\alpha$ of $Z$ is the variogram of order $\beta$ of $a^{\frac{\alpha}{\beta}} Y$.
\end{proof}

\begin{proof}[Proof of Theorem \ref{thm1b}]
The proof of the first representation is a direct extension of that of Theorem \ref{thm1} to the multivariate setting, the difference being that the mapping $\rho_{\omega}$ in Eq. (\ref{covar}) is a matrix-valued correlation function instead of a scalar correlation function. The second representation follows from the proof of Remark \ref{rem2} and from Eq. (\ref{vartocov}) for standard Gaussian random fields.
\end{proof}

\begin{proof}[Proof of Theorem \ref{thm2d}]
The proof is a direct extension of that of Theorem \ref{thm2c} to the multivariate setting. The sufficiency part relies on Proposition \ref{prop02b}. The necessity part rests on the fact that the mapping $(x,y) \mapsto [d^2_{\text{E}}(Z_i(x,\omega),Z_j(y,\omega))]_{i.j=1}^p$ fulfills the conditions of a pseudo-variogram indicated in Section \ref{pseud}, in particular, the negative-type inequalities (\ref{negtype}) \citep{Schoenberg1938}. 
\end{proof}

\begin{proof}[Proof of Theorem \ref{thm2b}]
This is a particular case of Theorem \ref{thm2d}.
\end{proof}

\bigskip

\noindent \textbf{Acknowledgments.} 
This work was funded and supported by the National Agency for Research and Development of Chile [grants ANID CIA250010 and ANID Fondecyt 1250008].

\bibliographystyle{elsarticle-num-names} 
\bibliography{mybib}

@article{Critchley1997,
  author   = {Critchley, Frank and Fichet, Bernard},
  title    = {On (super-) spherical distance matrices and two results from {S}choenberg},
  journal  = {Linear Algebra and its Applications},
  volume   = {251},
  pages    = {145--165},
  year     = {1997},
}

@inproceedings{Wilson2010,
  author   = {Wilson, R.C. and Hancock, E.R. and Pekalska, E. and Duin, R.P.W.},
  title    = {Spherical embeddings for non-{E}uclidean dissimilarities},
  booktitle    = {2010 {IEEE} {C}omputer {S}ociety {C}onference on {C}omputer {V}ision and {P}attern {R}ecognition},
  year         = {2010},
  pages        = {1903--1910},
  publisher  = {Institute of Electrical and Electronics Engineers},
}

@book{Lantuejoul2002,
  title={Geostatistical Simulation: {M}odels and Algorithms},
  author={Lantu\'ejoul, Christian},
  year={2002},
  publisher={Springer}
}

@book{Rokhlin1952,
  title={On the Fundamental Ideas of Measure Theory},
  author={Rokhlin, V.A.},
  year={1952},
  publisher={American Mathematical Society}
}

@book{Ito1984,
  title={Introduction to Probability Theory},
  author={It{\^o}, K.},
  year={1984},
  publisher={Cambridge University Press}
}

@book{Doob1953,
  title={Stochastic Processes},
  author={Doob, J.L.},
  year={1953},
  publisher={John Wiley \& Sons}
}

@article{Bogomolny2008,
  author   = {Bogomolny, E. and Bohigas, O. and Schmit, C.},
  title    = {Distance matrices and isometric embeddings},
  journal  = {Journal of Mathematical Physics, Analysis, Geometry},
  volume   = {4},
  number = {1},
  pages    = {7--23},
  year     = {2008},
}

@article{emery2005,
  author   = {Emery, X.},
  title    = {Variograms of order $\omega$: A tool to validate a bivariate distribution model},
  journal  = {Mathematical Geology},
  volume   = {37},
  number = {2},
  pages    = {163--181},
  year     = {2005},
}

@article{emeryortiz2011,
  author   = {Emery, X. and Ortiz, J.M.},
  title    = {A comparison of random field models beyond bivariate distributions},
  journal  = {Mathematical Geosciences},
  volume   = {43},
  number = {2},
  pages    = {183--202},
  year     = {2011},
}

@article{cressie1980,
  author   = {Cressie, N. and Hawkins, D.M.},
  title    = {Robust estimation of the variogram: I},
  journal  = {Mathematical Geology},
  volume   = {12},
  number = {2},
  pages    = {115--125},
  year     = {1980},
}

@article{Shamai1991,
  author   = {Shamai, S. and Chayat, N.},
  title    = {The autocorrelation function of a peak-power-limited process},
  journal  = {Signal Processing},
  volume   = {24},
  number = {2},
  pages    = {127--136},
  year     = {1991},
}

@article{emery2006,
  author   = {Emery, X. and Lantu\'ejoul, C.},
  title    = {{TBSIM}: {A} computer program for conditional simulation of three-dimensional {G}aussian random fields via the turning bands method},
  journal  = {Computers \& Geosciences},
  volume   = {32},
  number = {10},
  pages    = {1615--1628},
  year     = {2006},
}

@article{alegria2020,
  author   = {Alegr\'ia, A. and Emery, X. and Lantu\'ejoul, C.},
  title    = {The turning arcs: {A} computationally efficient algorithm to simulate isotropic vector-valued {G}aussian random fields on the $d$-sphere},
  journal  = {Statistics and Computing},
  volume   = {30},
  number = {5},
  pages    = {1403--1418},
  year     = {2020},
}

@article{naveau2009,
  author   = {Naveau, P. and Guillou, A. and Cooley, D. and Diebolt, J.},
  title    = {Modeling pairwise dependence of maxima in space},
  journal  = {Biometrika},
  volume   = {96},
  pages    = {1--17},
  year     = {2009},
}

@article{bacro2010,
  author   = {Bacro, J.N. and Bel, L. and Lantu\'ejoul, C.},
  title    = {Testing the independence of maxima: From bivariate vectors to spatial extreme fields},
  journal  = {Extremes},
  volume   = {13},
  pages    = {155--175},
  year     = {2010},
}

@article{Klein1993,
  author   = {Klein, D.J. and Randi{\'c}, M.},
  title    = {Resistance distance},
  journal  = {Journal of Mathematical Chemistry},
  volume   = {12},
  pages    = {81--95},
  year     = {1993},
}

@article{Klein1998,
  author   = {Klein, D.J. and Zhu, H.Y.},
  title    = {Distances and volumina for graphs},
  journal  = {Journal of Mathematical Chemistry},
  volume   = {23},
  pages    = {179--195},
  year     = {1998},
}

@article{Estrada2012,
  author   = {Estrada, E.},
  title    = {The communicability distance in graphs},
  journal  = {Linear Algebra and its Applications},
  volume   = {436},
  number   = {11},
  pages    = {4317--4328},
  year     = {2012},
}

@article{Anderes2020,
  author   = {Anderes, E. and M\"oller, J. and Rasmussen, J.G.},
  title    = {Isotropic covariance functions on graphs and their edges},
  journal  = {The Annals of Statistics},
  volume   = {48},
  number   = {4},
  pages    = {2478--2503},
  year     = {2020},
}

@article{Dorr2023,
  author   = {D\"orr, C. and Schlather, M.},
  title    = {Characterization theorems for pseudo cross-variograms},
  journal  = {Journal of Applied Probability},
  volume   = {60},
  number   = {4},
  pages    = {1219--1231},
  year     = {2023},
}

@Book{chiles_delfiner_2012,
  title         = {Geostatistics: Modeling Spatial Uncertainty},
  publisher     = {Wiley},
  year          = {2012},
  author        = {Chil{\`e}s, J.P. and Delfiner, P.},
  address       = {New York},
  edition       = {2nd},
}

@Book{cressie1993,
  title         = {Statistics for Spatial Data},
  publisher     = {Wiley},
  year          = {1993},
  author        = {Cressie, N.A.C.},
  address       = {New York},
  edition       = {2nd},
}

@Book{Matheron1975,
  title         = {Random Sets and Integral Geometry},
  publisher     = {Wiley},
  year          = {1975},
  author        = {Matheron, G.},
  address       = {New York},
}

@article{Schoenberg1935,
  author   = {Schoenberg, I.J.},
  title    = {Remarks to {M}aurice {F}r\'echet's article ``{S}ur la d\'efinition axiomatique d'une classe d'espace distanc\'es Vectoriellement applicable sur l'espace de {H}ilbert''},
  journal  = {Annals of Mathematics},
  volume   = {36},
  number = {3},
  pages    = {724--732},
  year     = {1935},
}

@article{Schoenberg1938,
  author   = {Schoenberg, I.J.},
  title    = {Metric spaces and positive definite functions},
  journal  = {Transactions of the American Mathematical Society},
  volume   = {44},
  number = {3},
  pages    = {522--536},
  year     = {1938},
}

@article{Galli2012,
  author   = {Galli, Laura and Kaparis, Konstantinos and Letchford, Adam N.J.},
  title    = {Complexity results for the gap inequalities for the max-cut problem},
  journal  = {Operations Research Letters},
  volume   = {40},
  pages    = {149--152},
  year     = {2012},
}

@inproceedings{Kelly1975,
  author   = {Kelly, J.B.},
  title    = {Hypermetric spaces},
  booktitle    = {The Geometry of Metric and Linear Spaces},
  editor = {Kelly, L.M.},
  year         = {1975},
  pages        = {17--31},
  publisher  = {Springer-Verlag},
}

@inproceedings{matheron1976,
  author   = {Matheron, G.},
  title    = {A simple substitute for conditional expectation: The disjunctive kriging},
  booktitle    = {Advanced Geostatistics in the Mining Industry},
  editor = {Guarascio, M. and David, M. and Huijbregts, C.},
  year         = {1976},
  pages        = {221-236},
  publisher  = {Springer},
}

@article{Laurent1996,
  author   = {Laurent, M. and Poljak, S.},
  title    = {Gap inequalities for the cut polytope},
  journal  = {European Journal of Combinatorics},
  volume   = {17},
  number   = {2--3},
  pages    = {233--254},
  year     = {1996},
}

@article{Barahona1986,
  author   = {Barahona, F. and Mahjoub, A.R.},
  title    = {On the cut polytope},
  journal  = {Mathematical Programming},
  volume   = {36},
  pages    = {157--173},
  year     = {1986},
}

@article{Laurent1995,
  author   = {Laurent, M. and Poljak, S.},
  title    = {On a positive semidefinite relaxation of the cut polytope},
  journal  = {Linear Algebra and its Applications},
  volume   = {223--224},
  pages    = {439--461},
  year     = {1995},
}

@article{Deza1961,
  author   = {Deza, M.},
  title    = {On the {H}amming geometry of unitary cubes},
  journal  = {Soviet Physics Doklady},
  volume   = {5},
  pages    = {940--943},
  year     = {1961},
}

@article{Letchford2012,
  author   = {Letchford, A.N. and S{\o}rensen, M.M.},
  title    = {Binary positive semidefinite matrices and associated integer polytopes},
  journal  = {Mathematical Programming},
  volume   = {131},
  pages    = {253--271},
  year     = {2012},
}

@article{Journel1983,
  author   = {Journel, A.G.},
  title    = {Nonparametric estimation of spatial distributions},
  journal  = {Mathematical Geology},
  volume   = {15},
  number   = {3},
  pages    = {445--468},
  year     = {1983},
}

@article{Emery2011,
  author   = {Emery, X. and Lantu\'ejoul, C.},
  title    = {Geometric covariograms, indicator variograms and boundaries of planar closed sets},
  journal  = {Mathematical Geosciences},
  volume   = {43},
  number   = {8},
  pages    = {905--927},
  year     = {2011},
}

@article{Emery2004,
  author   = {Emery, X.},
  title    = {Properties and limitations of sequential indicator simulation},
  journal  = {Stochastic Environmental Research and Risk Assessment},
  volume   = {18},
  number   = {6},
  pages    = {414--424},
  year     = {2004},
}

@article{Arroyo2012,
  author   = {Arroyo, D. and Emery, X. and Pel\'aez, M.},
  title    = {An enhanced Gibbs sampler algorithm for non-conditional simulation of Gaussian random vectors},
  journal  = {Computers \& Geosciences},
  volume   = {46},
  pages    = {138--148},
  year     = {2012},
}

@article{Maleki2017,
  author   = {Maleki, M. and Emery, X. and Mery, N.},
  title    = {Indicator variograms as an aid for geological interpretation and modeling of ore deposits},
  journal  = {Minerals},
  volume   = {7},
  number   = {12},
  pages    = {241},
  year     = {2017},
}

@Book{Polyanin,
  title         = {Handbook of Integral Equations},
  publisher     = {CRC Press},
  year          = {1998},
  author        = {Polyanin, A.D. and Manzhirov, A.V.},
  address       = {Boca Raton},
}

@Book{szego1975,
  title         = {Orthogonal Polynomials},
  publisher     = {American Mathematical Society},
  year          = {1939},
  author        = {Szeg\"o, G.},
  address       = {Providence, Rhode Island},
}

@Book{Serra1982,
  title         = {Image Analysis and Mathematical Morphology},
  publisher     = {Academic Press},
  year          = {1982},
  author        = {Serra, J.},
  address       = {London},
}

@Book{Chiu2013,
  title         = {Stochastic Geometry and its Applications},
  publisher     = {Wiley},
  year          = {2013},
  author        = {Chiu, S.N. and Stoyan, D. and Kendall, W.S. and Mecke, J.},
  address       = {New York},
}

@Book{Molchanov2017,
  title         = {Theory of Random Sets},
  publisher     = {Springer},
  year          = {2017},
  author        = {Molchanov, I.},
  address       = {London},
}

@article{Armstrong1992,
  author   = {Armstrong, M.},
  title    = {Positive definiteness is not enough},
  journal  = {Mathematical Geology},
  volume   = {24},
  number   = {1},
  pages    = {135--143},
  year     = {1992},
}

@article{Jeulin2000,
  author   = {Jeulin, D.},
  title    = {Random texture models for material structures},
  journal  = {Statistics and Computing},
  volume   = {10},
  pages    = {121--132},
  year     = {2000},
}

@Book{Jeulin2021,
  title         = {Morphological Models of random Structures},
  publisher     = {Springer Nature},
  year          = {2021},
  author        = {Jeulin, D.},
  address       = {Cham},
}

@inproceedings{Matheron1989,
  author   = {Matheron, G.},
  title    = {The internal consistency of models in geostatistics},
  booktitle    = {Geostatistics},
  editor = {Armstrong, M.},
  year         = {1989},
  pages        = {21--38},
  publisher  = {Kluwer Academic},
}

@inproceedings{Matheron1993,
  author   = {Matheron, G.},
  title    = {Une conjecture sur la covariance d'un ensemble al\'eatoire (A conjecture on the covariance of a random set)},
  booktitle    = {Cahiers de Géostatistique},
  editor = {de Fouquet, C.},
  year         = {1993},
  pages        = {107--113},
  publisher  = {Centre de Géostatistique, Ecole Nationale Supérieure des Mines de Paris, Fontainebleau},
}

@article{Lachieze2015,
  author   = {Lachi\`eze-Rey, R.},
  title    = {Realisability conditions for second-order marginals of biphased media},
  journal  = {Random Structures \& Algorithms},
  volume   = {47},
  pages    = {588--604},
  year     = {2015},
}

@inproceedings{Shepp1967,
  author   = {Shepp, L.A.},
  title    = {Covariances of unit processes},
  booktitle    = {Proceedings of the Working Conference on Stochastic Processes},
  year         = {1967},
  pages        = {205--218},
  address  = {Santa Barbara, California},
}

@article{McMillan1955,
  author   = {McMillan, B.},
  title    = {History of a problem},
  journal  = {Journal of the Society for Industrial and Applied Mathematics},
  volume   = {3},
  number   = {3},
  pages    = {114--128},
  year     = {1955},
}

@article{Quintanilla2008,
  author   = {Quintanilla, J.A.},
  title    = {Necessary and sufficient conditions for the two-point probability function of two-phase random media},
  journal  = {Proceedings of the Royal Society A: Mathematical, Physical and Engineering Sciences},
  volume   = {A464},
  pages    = {1761--1779},
  year     = {2008},
}

@article{Masry1972,
  author   = {Masry, E.},
  title    = {On covariance functions of unit processes},
  journal  = {SIAM Journal on Applied Mathematics},
  volume   = {23},
  pages    = {28--33},
  year     = {1972},
}

@article{Martins1983,
  author   = {Martins de Carvalho, J.L. and Clark, J.M.C.E.},
  title    = {Characterizing the autocorrelation of binary sequences},
  journal  = {IEEE Transactions on Information Theory},
  volume   = {29},
  number   = {4},
  pages    = {502--508},
  year     = {1983},
}

@Book{Torquato2002,
  title         = {Random Heterogeneous Materials: Microstructure and Macroscopic Properties},
  publisher     = {Springer},
  year          = {2002},
  author        = {Torquato, S.},
  address       = {New York},
}

@article{Jiao2008,
  author   = {Jiao, Y. and Stillinger, F.H.},
  title    = {Modeling heterogeneous materials via two-point correlation functions. {II} {A}lgorithmic details and applications},
  journal  = {Physical Review E},
  volume   = {77},
  pages    = {031135},
  year     = {2008},
}

@article{Yeong1998,
  author   = {Yeong, C.L.Y. and Torquato, S.},
  title    = {Reconstructing random media},
  journal  = {Physical Review E},
  volume   = {57},
  pages    = {495},
  year     = {1998},
}

@mastersthesis{Alabert1987,
    title    = {Stochastic imaging of spatial distributions using hard and soft information},
    school   = {Stanford University},
    address  = {Stanford},
    author   = {Alabert, F.},
    type = {Master's thesis},
    year     = {1987}, 
}

@article{Jones2001,
  author   = {Jones, T.A. and Ma, Y.Z.},
  title    = {Geologic characteristics of hole-effect variograms calculated from lithology-indicator variables},
  journal  = {Mathematical Geology},
  volume   = {33},
  number   = {5},
  pages    = {615--629},
  year     = {2001},
}

@article{Stein1988,
  author   = {Stein, M.},
  title    = {Asymptotically efficient prediction of a random field with a misspecified covariance function},
  journal  = {Annals of Statistics},
  volume   = {16},
  pages    = {55--63},
  year     = {1988},
}

@article{Western1998,
  author   = {Western, A.W. and Bloschl, G. and Grayson, R.B.},
  title    = {How well do indicator variograms capture the spatial connectivity of soil moisture?},
  journal  = {Hydrological processes},
  volume   = {12},
  number   = {12},
  pages    = {1851--1868},
  year     = {1998},
}

@article{Cao2014,
  author   = {Cao, R. and Ma, Y.Z. and Gomez, E.},
  title    = {Geostatistical applications in petroleum reservoir modelling},
  journal  = {Journal of the Southern African Institute of Mining and Metallurgy},
  volume   = {114},
  number   = {8},
  pages    = {625--629},
  year     = {2014},
}

@article{Guo2010,
  author   = {Guo, H. and Deutsch, C.V.},
  title    = {Fluvial channel size determination with indicator variograms},
  journal  = {Petroleum Geoscience},
  volume   = {16},
  number   = {2},
  pages    = {161--169},
  year     = {2010},
}

@article{He2009,
  author   = {He, Y. and Chen, D. and Li, B.G. and Huang, Y.F. and Hu, K.L. and Li, Y. and Willett, I.R.},
  title    = {Sequential indicator simulation and indicator kriging estimation of 3-dimensional soil textures},
  journal  = {Australian Journal of Soil Research},
  volume   = {47},
  number   = {6},
  pages    = {622--631},
  year     = {2009},
}

@article{Webster1994,
  author   = {Webster, R. and Atteia, O. and Dubois, J.P.R.},
  title    = {Coregionalization of trace-metals in the soil in the {S}wiss {J}ura},
  journal  = {European Journal of Soil Science},
  volume   = {45},
  number   = {2},
  pages    = {205--218},
  year     = {1994},
}

@article{Rossi1992,
  author   = {Rossi, R.E. and Mulla, D.J. and Journel, A.G. and Franz, E.H.},
  title    = {Geostatistical tools for modeling and interpreting ecological spatial dependence},
  journal  = {Ecological Monographs},
  volume   = {62},
  number   = {2},
  pages    = {277--314},
  year     = {1992},
}

@article{Dubrule2017,
  author   = {Dubrule, O.},
  title    = {Indicator variogram models: {D}o we have much choice?},
  journal  = {Mathematical Geosciences},
  volume   = {49},
  number   = {4},
  pages    = {441--465},
  year     = {2017},
}

@book{Olver,
	title        = {{NIST} {H}andbook of {M}athematical {F}unctions},
	author       = {Frank W.J. Olver and Daniel M. Lozier and Ronald F. Boisvert and Charles W. Clark},
	year         = 2010,
	publisher    = {Cambridge University Press, Cambridge}
}

@book{prud1,
	title        = {Integrals and Series, Volume 1: Elementary Functions},
	author       = {Prudnikov, A.P. and Brychkov, Y.A. and Marichev, O.I.},
	year         = 1986,
	publisher    = {Gordon and Breach}
}

@book{prud2,
	title        = {Integrals and Series, Volume 2: Special Functions},
	author       = {Prudnikov, A.P. and Brychkov, Y.A. and Marichev, O.I.},
	year         = 1986,
	publisher    = {Gordon and Breach}
}

@book{prud3,
	title        = {Integrals and Series, Volume 3: More Special Functions},
	author       = {Prudnikov, A.P. and Brychkov, Y.A. and Marichev, O.I.},
	year         = 1990,
	publisher    = {Gordon and Breach}
}

@book{prud4,
	title        = {Integrals and Series, Volume 4: Direct {L}aplace Transforms},
	author       = {Prudnikov, A.P. and Brychkov, Y.A. and Marichev, O.I.},
	year         = 1992,
	publisher    = {Gordon and Breach}
}

@inproceedings{Clark1989,
  author   = {Clark, I. and Basinger, K. and Harper, W.},
  title    = {{MUCK}---{A} novel approach to co-kriging},
  booktitle    = {Proceedings of the Conference on Geostatistical, Sensitivity, and Uncertainty: Methods for Ground-Water Flow and Radionuclide Transport Modeling},
  editor = {Buxton, B.E.},
  year         = {1989},
  pages        = {473--494},
  address    = {Columbus},
  publisher  = {Batelle Press},
}

@book{Grad,
  author =       {Gradshteyn, I.S. and Ryzhik, I.M.},
  publisher =    {Academic Press},
  title =        {Table of {I}ntegrals, {S}eries, and {P}roducts},
  year =         {2007},
  address = {Amsterdam},
  edition =      {7th.},
}

@misc{weisstein,
    author   = {Weisstein, Eric W.},
    title    = {Hypergeometric Functions. {From MathWorld---A Wolfram Web Resource}},
    url      = {\url{https://functions.wolfram.com/HypergeometricFunctions}},
    year = {2025},
    note     = {Last visited on 22/2/2025},
}

@book{Erdelyi1,
  author =       {Erd\'elyi, A. and Magnus, W. and Oberhettinger, F. and Tricomi, F.G.},
  publisher =    {McGraw Hill},
  title =        {Higher Transcerdental Functions. Volume 1},
  year =         {1953},
  address = {New York},
}

@book{Exton1978,
  title={Handbook of Hypergeometric Integrals},
  author={Exton, Harold},
  year={1978},
  address={Chichester},
  publisher={John Wiley \& Sons}
}

@article{Schoenberg1942,
	title        = {Positive definite functions on spheres},
	author       = {Schoenberg, I.J.},
	year         = 1942,
	journal      = {Duke Mathematical Journal},
	publisher    = {Duke University Press},
	volume       = 9,
	number       = 1,
	pages        = {96--108}
}

@article{Tabor2009,
	title        = {Takagi functions and approximate midconvexity},
	author       = {Tabor, J. and Tabor, J.},
	year         = 2009,
	journal      = {Journal of Mathematical Analysis and Applications},
	publisher    = {Elsevier},
	volume       = 356,
	number       = 2,
	pages        = {729--737}
}

@article{Matheron1988,
	title        = {Simulation de fonctions al\'eatoires admettant un variogramme concave donn\'e (Simulation of random functions having a given concave variogram)},
	author       = {Matheron, G.},
	year         = 1988,
	journal      = {Sciences de la Terre, S\'erie Informatique g\'eologique},
	publisher    = {Elsevier},
	volume       = 28,
	pages        = {195--212}
}

@book{Oldham2009,
  title={An Atlas of Functions},
  author={Oldham, K. and Myland, J. and Spanier, J.},
  year={2009},
  address={New York},
  publisher={Springer}
}

@article{Matheron1973,
	title        = {The intrinsic random functions and their applications},
	author       = {G. Matheron},
	year         = 1973,
	journal      = {Advances in Applied Probability},
	volume       = 5,
	pages        = {439--468}
}

@misc{math75,
	title        = {Compl\'ements sur les mod\`eles isofactoriels},
	author       = {Matheron, G.},
	year         = 1975,
    note         = {Internal Report N-432},
	publisher    = {Centre de G\'eostatistique, Ecole Nationale Sup\'erieure des Mines de Paris}
}

@misc{Shepp1963,
	title        = {On positive-definite functions associated with certain stochastic processes},
	author       = {Shepp, L.A.},
	year         = 1963,
    note         = {Technical report 63-1213-11},
	publisher    = {Bell Laboratories},
    address      = {Murray Hill, NJ}
}

\end{document}